\documentclass[12pt]{article}
\usepackage{amsmath,amssymb}
\usepackage{amscd}
\usepackage{comment}
\usepackage{color}
\usepackage{enumerate}
\usepackage{cases}
\usepackage{mathtools}
\usepackage[all]{xy}
\usepackage{amsthm}
\usepackage{pb-diagram}
\usepackage[top=25truemm,bottom=25truemm,left=15truemm,right=15truemm]{geometry}
\theoremstyle{definition}
\newtheorem{defi}{Definition}
\newtheorem{theo}[defi]{Theorem}
\newtheorem{lemm}[defi]{Lemma}
\newtheorem{prop}[defi]{Proposition}

\newtheorem{rem}[defi]{Remark}
\newtheorem{que}[defi]{Question}
\newtheorem{cond}[defi]{Condition}

\def\det{{\rm det}}

\def\SO{{\rm SO}}
\def\SL{{\rm SL}}
\def\SU{{\rm SU}}

\def\End{{\rm End}}

\def\det{{\rm det}}

\def\Vol{{\rm Vol}}

\def\refe{\ast}

\def\vol{{\rm vol}}
\def\can{{\rm can}}
\def\max{{\rm max}}
\def\min{{\rm min}}

\def\R{{\mathbb R}}
\def\Z{{\mathbb Z}}
\def\C{{\mathbb C}}
\def\N{{\mathbb N}}

\def\K{{\mathbb K}}
\def\CP{{\mathbb C}{\mathbb P}}
\def\D{{\mathbb D}}

\def\inum{{\sqrt{-1}}}

\providecommand{\keywords}[1]
{
\small 
\textbf{{Keywords---}} #1
}
\providecommand{\MSC}[1]
{
\small 
\textbf{{MSC---}} #1
}

\begin{document}

\title {Shannon entropy for harmonic metrics \\
on cyclic Higgs bundles II
}
\author {Natsuo Miyatake}
\date{}
\maketitle
\begin{abstract} 
Let $X$ be a Riemann surface and $K_X \rightarrow X$ the canonical bundle. For each integer $r \geq 2$, each $q \in H^0(K_X^r)$, and each choice of the square root $K_X^{1/2}$ of the canonical bundle, we canonically obtain a Higgs bundle, which is called a cyclic Higgs bundle. A diagonal harmonic metric $h = (h_1, \dots, h_r)$ on a cyclic Higgs bundle yields $r-1$-Hermitian metrics $H_1, \dots, H_{r-1}$ on $K_X^{-1} \rightarrow X$, defined as $H_j \coloneqq h_j^{-1} \otimes h_{j+1}$ for each $j=1,\dots, r-1$, while $h_1$, $h_r$, and $q$ yield a degenerate Hermitian metric $H_r$ on $K_X^{-1} \rightarrow X$. 
The $r$-differential $q$ induces a subharmonic weight function $\phi_q=\frac{1}{r}\log|q|$ on $K_X\rightarrow X$, and the diagonal harmonic metric depends solely on this weight function. In the previous papers, the author introduced and studied the extension of harmonic metrics associated with arbitrary subharmonic weight function $\varphi$, which also constructs $r-1$-Hermitian metrics $H_1,\dots, H_{r-1}$ and a degenerate Hermitian metric $H_r$ on $K_X^{-1}\rightarrow X$. Especially, the author introduced a function called entropy that quantifies the degree of mutual misalignment of the Hermitian metrics $H_1,\dots, H_r$. In this paper, by analogy with the canonical ensemble in equilibrium statistical mechanics, we further introduce the quantity which we call free energy. 
When $H_1,\dots, H_{r-1}$ are all complete and satisfy a condition concerning their approximation, we give a sufficient condition for the free energy to decrease at each point, and when $r=2,3$ we also give a sufficient condition for the entropy to increase at each point. 
Furthermore, on the unit disc $\D\coloneqq \{z\in\C\mid |z|<1\}$, when $e^\varphi$ is $C^2$ outside a compact subset, we provide, from the perspective of entropy and free energy, necessary and sufficient conditions for the function $e^\varphi h_\refe^{-1} \otimes h_\D$ to be bounded, where $h_\D$ denotes the Hermitian metric on the canonical bundle induced by the Poincar\'e metric, and $h_\refe$ is a reference metric. This result extends the work of Wan, Benoist-Hulin, Labourie-Toulisse, and Dai-Li. In particular, we introduce a new concept called the redundancy function and the lower redundancy constant, which quantifies the deviation of the entropy from the maximum entropy, and show that lower redundancy constant being positive is equivalent to $e^\varphi h_\refe^{-1}\otimes h_\D$ being bounded.
\end{abstract}

\MSC{30C15, 31A05, 53C07}


\section{Introduction}\label{1}
This paper is a continuation of the paper \cite{Miy4} (see also \cite{Miy2, Miy3, Miy5}). In the previous paper, we introduced the notion of entropy for harmonic metrics on cyclic Higgs bundles, where the Higgs field may have multivalency of cyclic type, and for Hermitian metrics that solve the Hitchin equation with a ``collapsed" coefficient obtained as the limit of infinitely increasing covering order. In this paper, we continue to investigate the properties of entropy and introduce a quantity called {\it free energy}, in analogy with the theory of canonical ensemble in statistical physics. Unlike entropy, free energy is 
related to existing quantities that have been well studied so far: the energy density of harmonic maps, or the norm of the Higgs field. One of the aims of this paper is to establish both an extension of existing results on the energy density of harmonic maps 
and a new interpretation of them.

A cyclic Higgs bundle is a Higgs bundle associated with a holomorphic section of a power of the canonical bundle on a Riemann surface (cf. \cite{Bar1, Hit1, Hit2}). If, for example, the base Riemann surface is a compact hyperbolic Riemann surface, then there exists a unique harmonic metric with fixed determinant on every cyclic Higgs bundle. Cyclic Higgs bundles and harmonic metrics thereon are special cases of more general Higgs bundles and harmonic metrics (cf. \cite{Cor1, DO1, Don0, Hit1, Sim1}), but in this paper we do not regard them as such, but rather emphasize our view and stance that they are natural extensions of nonnegative constant curvature K\"ahler metrics, in the following sense: We consider the rank 2 cyclic Higgs bundle over a compact Riemann surface for simplicity. Let $X$ be a compact Riemann surface and $K_X\rightarrow X$ the canonical bundle. We choose a square root $K_X^{1/2}$ of the canonical bundle and define a rank $2$-holomorphic vector bundle $\K_2\rightarrow X$ as follows:
\begin{align*}
\K_2\coloneqq K_X^{1/2}\oplus K_X^{-1/2}.
\end{align*}
Then for each $q$, we can associate a holomorphic section $\Phi(q)$ of $\End \K_2\otimes K_X\rightarrow X$ as follows:
\begin{align*}
&\Phi(q)\coloneqq \left(
\begin{array}{cc}
0 & q \\ 
1 & 0
\end{array}
\right).
\end{align*}
This pair $(\K_2,\Phi(q))\rightarrow X$ is a cyclic Higgs bundle of rank $2$, which was introduced in Hitchin's frist paper \cite[Section 11]{Hit1}. Suppose that $X$ is a hyperbolic Riemann surface. Then $(\K_2,\Phi(q))\rightarrow X$ is always stable, and thus from the Kobayashi-Hitchin correspondence there uniquely exists a harmonic metric $h$ satisfying $\det(h)=1$. Moreover, from the symmetry of the Higgs field and the uniqueness of the harmonic metric, the harmonic metric $h$ splits as $h=(h_1,h_1^{-1})$. Let $T_X\rightarrow X$ be the dual bundle of the canonical bundle $K_X\rightarrow X$. The solution $h=(h_1,h_1^{-1})$ induces a metric $H_1\coloneqq h_1^{-2}$ on $T_X\rightarrow X$ and one degenerate metric $H_2\coloneqq h_1^2\otimes h_q$ on $T_X\rightarrow X$, where $h_q$ denotes the degenerate Hermitian metric on $K_X^{-2}\rightarrow X$ induced by $q$, defined as follows:
\begin{align}
(h_q)_x(u,v)\coloneqq \langle q_x,u\rangle\overline{\langle q_x,v\rangle} \ \text{for $x\in X$, $u,v\in (K_X^{-2})_x$}. \label{uv}
\end{align}
If $q$ is identically zero, then the K\"ahler metric induced by $H_1$ is a constant negative Gaussian curvature metric. Suppose that $X$ is an elliptic curve. In this case, the harmonic metric on $(\K_2,\Phi(q))\rightarrow X$ exists if and only if $q\neq 0$, and if $q\neq 0$, then $h=(h_1, h_1^{-1})=(h_q^{-1/4}, h_q^{1/4})$ is a harmonic metric with $\det(h)=1$, where $h_q$ is a Hermitian metric on $K_X^{-2}\rightarrow X$ defined similarly as (\ref{uv}). If we define $H_1=h_1^{-2}$ and $H_2=h_1^2\otimes h_q$ in the same way as in the case where $X$ is hyperbolic, then we obtain $H_1=H_2=h_q^{1/2}$, and the K\"ahler metrics induced by them are flat metrics. Similar observations can be made for cyclic Higgs bundles of higher ranks as well, as will be explained in Section \ref{2}. In the sense described above, a diagonal harmonic metric on a cyclic Higgs bundle is an extension of a K\"ahler metric of nonnegative constant curvature associated with a holomorphic section of a power of the canonical bundle that contain flat metrics and hyperbolic metrics as two extreme cases: the case where $q$ has no zeros at all, and the case where $q$ is identically zero. For a general harmonic metric, the deviation from the flat metric occurs due to the zero points of the holomorphic section, and in the most extreme case where the section is a zero section, the harmonic metric is essentially equivalent to the hyperbolic metric.

In previous papers \cite{Miy2, Miy3, Miy4, Miy5}, the author proposed a new direction for the study of harmonic metrics on cyclic Higgs bundles in connection with the potential theory (cf. \cite{Ran1, ST1}), in particular the mathematics related to the distribution of zeros in holomorphic sections (cf. \cite{BCHM1, Ran1, ST1, SZ1}). In addition to the usual harmonic metrics on cyclic Higgs bundles, he proposed to consider harmonic metrics on cyclic ramified coverings and Hermitian metrics obtained that solves an extended Hitchin equation obtained as the limit as the covering order going to infinity. 
This is based on the observation that the harmonic metric on an ordinary cyclic Higgs bundle depends on the Hermitian metric $h_q$ induced by $q$ rather than on $q$ itself, and the extended Hitchin equation is nothing but the generalization from the singular metric $h_q^{-1/r}$ to a more general semipositive singular metric $e^{-\varphi}h_\refe$ on the canonical bundle. As with the usual harmonic metric on a cyclic Higgs bundle, a solution $h=(h_1,\dots, h_r)$ to the extended Hitchin equation associated with a singular metric $e^{-\varphi}h_\refe$ on the canonical bundle yields $r-1$ Hermitian metrics $H_1,\dots, H_{r-1}$ and one Hermitian metric $H_r$ which degenerates at the $-\infty$ set of $\varphi$. In the previous paper \cite{Miy4}, a function called {\it entropy}, obtained by combining $H_1,\dots, H_r$, was introduced. In \cite{Miy4}, it was presented that a uniform estimate of entropy from above and below and the limit of the difference between the upper and lower bounds when $r$ becomes infinitely large. In this paper, we introduce a new quantity called {\it free energy}, pursuing an analogy to the canonical ensemble in statistical physics. As mentioned above, it is closely related to the energy density of harmonic maps, or the norm of the Higgs field. Both entropy and free energy are defined by combining $H_1,\dots, H_r$. In the following, when we say ``$h$ is a complete solution associated with $\varphi$", we mean that $h$ is a solution to the extended Hitchin equation associated with $\varphi$ described in Section \ref{2.2} and it is complete in the sense described in Section \ref{rdc}. 
The following theorems on entropy and free energy for complete metrics are the main theorems of this paper:

\begin{theo}[Theorem \ref{main theorem b}]\label{main theorem 2}{\it Let $e^{-\varphi}h_\refe$ and $e^{-\varphi^\prime}h_\refe$ be semipositive singular Hermitian metrics on $K_X \to X$. Suppose that there exist complete solutions $h$ and $h^\prime$ associated with $\varphi$ and $\varphi^\prime$, respectively, which satisfy Condition \ref{three} in Section \ref{2.3}. Consider the following conditions on $\varphi$ and $\varphi^\prime$:
\begin{enumerate}[(i)]
\item For every $x\in X$, $\varphi(x)\geq \varphi^\prime(x)$, and for at least one $x\in X$, $\varphi(x)>\varphi^\prime(x)$.
\item For every $x\in X$, $\inum\bar{\partial}\partial \varphi(x)\geq \inum\bar{\partial}\partial\varphi^\prime(x)$.
\item When the mollification $(\widetilde{\varphi}_\epsilon)_{\epsilon>0}$ of the pullback of $\varphi$ to the universal covering is taken as in Condition \ref{three}, the metric $e^{\widetilde{\varphi}_\epsilon}\widetilde{h}_\refe^{-1}$ is complete for each $\epsilon>0$.
\end{enumerate}
Suppose $\varphi$ and $\varphi^\prime$ satisfy condition (i), $r=2$ or $3$, and $\beta$ is positive. Then the free energy functions for $h$ and $h^\prime$ satisfy the following inequality:
\begin{align*}
0 < F(r,\beta,\varphi^\prime, H_\refe)(x) - F(r,\beta,\varphi,H_\refe)(x) \ \text{for each $x\in X$}.
\end{align*}
Suppose that $\varphi$ and $\varphi^\prime$ satisfy conditions (i), (ii), and (iii), and that $\beta$ is positive. Then for any $r$, the free energy functions for $h$ and $h^\prime$ satisfy:
\begin{align*}
0 < F(r,\beta,\varphi^\prime, H_\refe)(x) - F(r,\beta,\varphi,H_\refe)(x) < r(\varphi(x) - \varphi^\prime(x)) + \frac{1}{\beta}\log r \ \text{for each $x\in X$}.
\end{align*}
}
\end{theo}

\begin{theo}[Theorem \ref{main theorem a}]\label{main theorem 1}{\it Let $r$ be either $2$ or $3$. Let $e^{-\varphi}h_\refe$ and $e^{-\varphi^\prime}h_\refe$ be semipositive singular Hermitian metrics on $K_X \to X$. Suppose that there exist complete solutions $h$ and $h^\prime$ associated with $\varphi$ and $\varphi^\prime$, respectively, which satisfy Condition \ref{three} in Section \ref{2.3}. Suppose also that $\varphi$ and $\varphi^\prime$ satisfy the following:
\begin{itemize}
\item For each $x \in X$, $\varphi(x) \geq \varphi^\prime(x)$.
\item For each $x \in X$, $\inum\bar{\partial}\partial \varphi(x) \geq \inum\bar{\partial}\partial \varphi^\prime(x)$.
\item When the mollification $(\widetilde{\varphi}_\epsilon)_{\epsilon>0}$ of the pullback of $\varphi$ to the universal covering is taken as in Condition \ref{three}, the metric $e^{\widetilde{\varphi}_\epsilon}\widetilde{h}_\refe^{-1}$ is complete for each $\epsilon>0$.
\end{itemize}
Then the entropy functions for $h$ and $h^\prime$ satisfy
\begin{align}
S(r,\beta,\varphi)(x) \geq S(r,\beta,\varphi^\prime)(x) \ \text{for all } x \in X. \label{S(x)}
\end{align}
If, moreover, $\varphi^\prime - \varphi$ belongs to $W^{1,2}_{loc}$ and $\varphi(x) > \varphi^\prime(x)$ holds at least at one point of $X$, then inequality (\ref{S(x)}) becomes strict for all points of $X$.

}
\end{theo}

\begin{theo}[Theorem \ref{main theorem c}]\label{main theorem 3}{\it Suppose that $X$ is the unit disc $\D\coloneqq \{z\in\C\mid |z|<1\}$. Let $h_X$ be the Hermitian metric on $K_X\rightarrow X$ induced by the Poincar\'e metric. Let $e^{-\varphi}h_X$ be a semipositive singular Hermitian metric on the canonical bundle satisfying the following condition:
\begin{itemize}
\item There exists a compact subset $K\subseteq X$ such that on $X\backslash K$, $e^{\varphi}$ is of class $C^2$.
\end{itemize}
We consider the entropy function and the free energy function constructed from the unique complete solution associated with $\varphi$. We set a constant $M_\varphi$, which may be infinite, by $M_\varphi\coloneqq \sup_X e^\varphi$. Suppose further that $e^\varphi$ belongs to $W^{1,2}_{loc}$ when $r=2,3$. Then the following are equivalent:
\begin{enumerate}[(a)]
\item $e^\varphi$ is bounded above, i.e., $M_\varphi$ is finite.
\item For each non-zero real constant $\beta$ and for each $r\geq 2$, there exists a positive constant $\delta$ that such that the following uniform estimate holds:
\begin{align*}
S(r,\beta,\varphi)\leq \log r-\delta.
\end{align*}
\item The following lower redundancy constant $\underline{R}_{r,\beta,\varphi}$ is positive:
\begin{align*}
\underline{R}_{r,\beta,\varphi}\coloneqq \inf_{x\in X}(1-(S(r,\beta,\varphi)(x)/\log r)).
\end{align*}
\item For each positive constant $\beta$ and for each $r\geq 2$, there exists a positive constant $C_1$ such that the following uniform estimate holds:
\begin{align*}
F(r,\beta,-\infty, H_\refe)-F(r,\beta,\varphi, H_\refe)\leq C_1.
\end{align*}
\item For each negative constant $\beta$ and for each $r\geq 2$, there exists a positive constant $C_2$ such that the following uniform estimate holds:
\begin{align*}
F(r,\beta,\varphi, H_\refe)-F(r,\beta,-\infty, H_\refe)\geq -C_2.
\end{align*}
\item For each positive constant $\beta$ and for each $r\geq 2$, there exists a positive constant $C_3$ such that the following uniform estimate holds:
\begin{align*}
\inum \partial \bar{\partial} (F(r,\beta,\varphi, H_\refe)-F(r,\beta,-\infty, H_\refe))\leq -C_3 e^{-F(r,\beta,\varphi, H_\refe)}\vol(H_\refe)+\omega_X,
\end{align*}
where $\omega_X$ is the K\"ahler form of $h_X$.
\end{enumerate}
}
\end{theo}

Roughly speaking, Theorem \ref{main theorem 2} say that the more singular $e^{-\varphi}h_\refe$ is, the smaller the free energy is. In thermodynamics, the variational principle is well known, which states that an equilibrium state of the entire system is achieved if and only if the free energy of the entire system, which is the sum of the local equilibrium free energies, is minimized (cf. \cite{Cal1}). On the other hand, equilibrium weight (cf. \cite{BB1, BBN1, BL1, BL2, BL3, Kli1, ST1}) is a well-known traditional thermodynamic concept in potential theory. In rather rough terms that omit the finer details, the equilibrium weight is defined as the usc regularization of the supremum of psh weight functions constrained by a fixed continuous weight function (cf. \cite{BB1, BBN1}). It seems interesting that Theorem \ref{main theorem 2} is roughly consistent with the variational principle of free energy in thermodynamics and the concept of equilibrium weight. It remains to be seen in the future to explore further connections between the concepts we have defined and existing concepts in potential theory, such as equilibrium weights. 

Because the theory concerning the existence, uniqueness, and approximation of a complete solution associated with $\varphi$ remains incomplete (see Section \ref{2.3}), Theorems \ref{main theorem 2}, \ref{main theorem 1}, and \ref{main theorem 3} each include conditions on $\varphi$ and the complete solution associated with it. However, the author expects that a complete solution associated with $\varphi$ always exists as long as $\varphi$ is not identically $-\infty$, and that approximation in the sense of Condition \ref{three} can always be achieved.

In the proofs of Theorems \ref{main theorem 2} and \ref{main theorem 3}, we use the techniques developed by Dai-Li \cite{DL2, DL3}.
Theorem \ref{main theorem 2} above is an extension of the monotonicity theorem for pullback metrics of harmonic maps proved in \cite{DL2} for the case where $\varphi^\prime=\phi_q$ and $\varphi=\phi_q+\log t \ (t>0)$, and Theorem \ref{main theorem 3} above is a new formulation of the result of Dai-Li \cite{DL3} (see also \cite{BH1, LT1, Wan1} for the lower rank cases) in terms of entropy and free energy, and an extension of their result to the case of more general subharmonic weight function $\varphi$. For more details on the results of Dai-Li, see Sections \ref{2.2} and \ref{2.4}.

The constants $\delta, C_1,$ $C_2$, and $C_3$ in Theorem \ref{main theorem 3} can all be chosen to depend only on $M_\varphi, \beta,$ and $r$, provided that $e^\varphi$ is globally of class $C^2$ (see Remark \ref{phiremark}).

The assumption in Theorem \ref{main theorem 1} that $\varphi^\prime - \varphi$ belongs to $W^{1,2}_{loc}$ and the assumption in Theorem \ref{main theorem 3} that $e^{\varphi}h_\refe^{-1}$ belongs to $W^{1,2}_{loc}$ are made to apply the mean value inequality (see Section \ref{CYMMVI}). Although these assumptions may seem unnecessary, the proof is not yet complete without them. 


Note that Theorem \ref{main theorem 1} is conditional on the cases $r=2,3$. For the cases $r\geq 4$, the author suspects that it does not hold the same statement as in Theorem \ref{main theorem 2}. 

This paper is organized as follows: In Section \ref{2}, we review the background materials relevant to the present paper. In Section \ref{CYMMVI}, we will give a brief review of the Cheng-Yau maximum principle and the mean value inequality, which will be used in the proofs of the main theorems. In Section \ref{3}, we introduce the notion of free energy. In Section \ref{4}, we introduce the notion of the redundancy function and the lower redundancy constant. In Section \ref{5}, we state our main theorems. Finally, in Section \ref{6}, we provide proofs of the main theorems. 

\noindent
{\bf Acknowledgements.} I would like to thank Toshiaki Yachimura for helpful conversations. 
I would like to express my gratitude to Qiongling Li for her valuable comments on my previous work and for introducing me to her paper \cite{DL3} when I sent her the first draft of my previous paper \cite{Miy3}. Her insightful feedback played a crucial role in the development of Theorem \ref{main theorem 3}. This work was supported by the Grant-in-Aid for Early Career Scientists (Grant Number: 24K16912).

\section{Backgrounds}\label{2}
\subsection{Real, diagonal, and complete Hermitian metrics}\label{rdc}
Let $X$ be a connected, possibly non-compact Riemann surface and $K_X\rightarrow X$ the canonical bundle. We choose a square root $K_X^{1/2}$ of the canonical bundle $K_X\rightarrow X$. Let $\K_r$ be a holomorphic vector bundle of rank $r$ defined as follows:
\begin{align*}
\K_r\coloneqq \bigoplus_{j=1}^r K_X^{\frac{r-(2j-1)}{2}} = K_X^{\frac{r-1}{2}} \oplus K_X^{\frac{r-3}{2}} \oplus \cdots \oplus K_X^{-\frac{r-3}{2}} \oplus K_X^{-\frac{r-1}{2}}.
\end{align*}
A Hermitian metric $h$ on $\K_r\rightarrow X$ is said to be {\it real} (cf. \cite{Hit2}) if the following bundle isomorphism $S$ is isometric with respect to $h$:
\begin{align*}
S\coloneqq \left(
\begin{array}{ccc}
& & 1 \\ 
& \reflectbox{$\ddots$} & \\ 
1 & &
\end{array}
\right): \K_r\rightarrow \K_r^{\vee},
\end{align*}
where $\K_r^{\vee}$ denotes the dual bundle of $\K_r$. Also, $h$ is called {\it diagonal} (cf. \cite{Bar1, Col1, DL2}) if it splits as $h=(h_1,\dots, h_r)$. Let $T_X\rightarrow X$ be the dual bundle of $K_X\rightarrow X$. By taking the difference between the adjacent components in a diagonal Hermitian metric $h=(h_1,\dots, h_r)$, we obtain $r-1$-Hermitian metrics $H_1,\dots, H_{r-1}$ on $T_X\rightarrow X$:
\begin{align*}
H_j\coloneqq h_j^{-1}\otimes h_{j+1}.
\end{align*}
A diagonal Hermitian metric $h$ is real if and only if $h_j=h_{r-j+1}$ for all $j=1,\dots, r$. This is also equivalent to the metrics $H_1,\dots, H_{r-1}$ satisfying $H_j=H_{r-j}$ for all $j=1,\dots, r-1$. We say that a diagonal Hermitian metric $h$ is {\it complete} (cf. \cite{LM1}) if the K\"ahler metrics induced by $H_1,\dots, H_{r-1}$ are all complete. 
\subsection{Harmonic metrics on cyclic Higgs bundles}\label{2.1}
Let $q$ be a holomorphic section of $K_X^r\rightarrow X$. Then we associate a holomorphic section $\Phi(q)\in H^0(\End \K_r\otimes K_X)$, which is called a Higgs field, as follows:
\begin{align*}
\Phi(q)\coloneqq 
\left(
\begin{array}{cccc}
0 & && q \\ 
1 & \ddots && \\ 
& \ddots & \ddots & \\ 
& & 1 & 0
\end{array}
\right).
\end{align*}
The pair $(\K_r,\Phi(q))$ is called a cyclic Higgs bundle (cf. \cite{Bar1, Hit1, Hit2}). The above cyclic Higgs bundle is an example of cyclotomic Higgs bundles introduced in \cite{Sim3}. For each Hermitian metric $h$ on $\K_r$, we associate a connection $D_q(h)$ defined as follows:
\begin{align*}
D_q(h)\coloneqq \nabla^h + \Phi(q) + \Phi(q)^{\ast h}.
\end{align*}
A Hermitian metric on $(\K_r,\Phi(q))\rightarrow X$ is called a {\it harmonic metric} if the connection $D_q(h)$ is a flat connection. A harmonic metric $h$ yields a harmonic map $\hat{h}: \widetilde{X}\rightarrow \SL(r,\C)/\SU(r)$, where $\widetilde{X}$ is the universal covering space. If $h$ is real, then $\hat{h}$ maps into $\SL(r,\R)/\SO(r)$. Suppose that $X$ is a compact Riemann surface of genus at least $2$. Then the Higgs bundle $(\K_r,\Phi(q))$ is stable for any $q\in H^0(K_X^r)$, and thus there uniquely exists a harmonic metric $h$ on $(\K_r,\Phi(q))$ such that $\det(h)=1$. From the uniqueness of the metric and the symmetry of the Higgs field (cf. \cite{Bar1, Hit2}), the harmonic metric $h$ is diagonal and real. A non-diagonal harmonic metric (resp. a non-real harmonic metric) can be constructed by considering the Dirichlet problem (cf. \cite{Don1, LM2}) for the Hitchin equation with a non-diagonal (resp. a non-real) Hermitian metric on the boundary. 
Let $h=(h_1,\dots, h_r)$ be a diagonal harmonic metric on a cyclic Higgs bundle $(\K_r,\Phi(q))\rightarrow X$. As mentioned in the previous subsection we obtain $r-1$-Hermitian metrics $H_1,\dots, H_{r-1}$ on $T_X\rightarrow X$ defined as
\begin{align*}
H_j\coloneqq h_j^{-1}\otimes h_{j+1} \ \text{for each $j=1,\dots,r -1$}.
\end{align*}
We also obtain a Hermitian metric $H_r$ on $T_X\rightarrow X$ which is degenerate at the zeros of $q$ as follows:
\begin{align*}
H_r\coloneqq h_r^{-1}\otimes h_1 \otimes h_q,
\end{align*}
where $h_q$ is a Hermitian metric on $K_X^{-r}\rightarrow X$ that is degenerate at the zeros of $q$ defined as follows:
\begin{align}
(h_q)_x(u,v)\coloneqq \langle q_x,u\rangle\overline{\langle q_x,v\rangle} \ \text{for $x\in X$, $u,v\in (K_X^{-r})_x$}. \label{q}
\end{align}
The connection $D_q(h)$ is flat if and only if the Hermitian metric $h$ satisfies the following elliptic equation, called the Hitchin equation:
\begin{align*}
F_h + [\Phi(q)\wedge\Phi(q)^{\ast h}] = 0.
\end{align*}
By using $H_1,\dots, H_r$, we can describe the Hitchin equation as follows: 
\begin{align}
\inum F_{h_j} + \vol(H_{j-1}) - \vol(H_j) = 0 \ \text{for $j=1,\dots, r-1$}, \label{F}
\end{align}
where for each $j=1,\dots, r$, $\vol(H_j)$ is the volume form of the metric $H_j$ which is locally described as $\vol(H_j)=\inum H_j(\frac{\partial}{\partial z}, \frac{\partial}{\partial z})\ dz\wedge d\bar{z}$, and $H_0$ is $H_r$. Equation (\ref{F}) is a kind of Toda equation (cf. \cite{AF1, Bar1, GH1, GL1, Moc0, Moc1}). By making a variable change from $h=(h_1,\dots, h_r)$ to $H_1,\dots, H_{r-1}$, the Hitchin equation becomes the following:
\begin{align}
\inum F_{H_j} + 2\vol(H_j) - \vol(H_{j-1}) - \vol(H_{j+1}) = 0 \ \text{for $j=1,\dots, r-1$.} \label{cyc}
\end{align}
In the case where $r=2$ and $q=0$, the Hitchin equation for the cyclic Higgs bundle is equivalent to the following equation:
\begin{align}
\inum F_H + 2\vol(H) = 0, \label{H}
\end{align}
where the solution $H$ to equation (\ref{H}) is a Hermitian metric $H$ on $T_X\rightarrow X$. The K\"ahler metric induced by $H$ is a constant negative Gaussian curvature K\"ahler metric (cf. \cite{Hit1}). 
We remark on the following two points:
\begin{itemize}
\item If $q$ has no zeros, then $H_1=\cdots=H_{r-1}=h_q^{1/r}$ is a solution to PDE (\ref{cyc}).
\item Suppose that $q=0$ and that there is a solution $H$ to equation (\ref{H}). We define positive constants $\lambda_1,\dots, \lambda_{r-1}$ as follows: 
\begin{align}
\lambda_j\coloneqq 2\sum_{k=1}^{r-1}(\Lambda_{r-1}^{-1})_{jk} = j(r-j), \label{lambda}
\end{align}
where $\Lambda_{r-1}^{-1}$ is the inverse matrix of the Cartan matrix of type $A_{r-1}$. Then $H_j=\lambda_j H\ (j=1,\dots, r-1)$ is a solution to PDE (\ref{cyc}).
\end{itemize}

Li-Mochizuki \cite{LM1} established the following result:
\begin{theo}[\cite{LM1}]{\it Suppose that $q$ is a non-zero holomorphic section unless $X$ is hyperbolic. Then there always exists a unique complete harmonic metric $h$ on $(\K_r,\Phi(q))\rightarrow X$ satisfying $\det(h)=1$. Moreover, the complete harmonic metric $h$ is real.
}
\end{theo}
In addition to \cite{LM1}, we refer the reader to \cite{DL3, DW1, GL1, LT1, LTW1, Li1, LM2, Moc0, Moc1, Nie1, Wan1, WA1} for the works related to complete harmonic metrics. 

\subsection{Dai-Li maximum principle}\label{2.2}
In this section, we review the maximum principle techniques for systems of PDEs established in \cite[Lemma 3.1]{DL1}.
Let $g$ be a K\"ahler metric on $X$. Adapting the notation of \cite{DL2}, we define $\Delta_g$ as the negative Laplacian, i.e.
\begin{align*}
\Delta_g\coloneqq 2\Lambda_g\partial \bar{\partial}=-d^\ast d,
\end{align*}
where we denote by $\Lambda_g$ the contraction operator induced by the K\"ahler metric $g$. 
Dai-Li established the following in \cite[Lemma 3.1]{DL1}:
\begin{theo}[\cite{DL1}]\label{DLM}{\it For each $j=1,\dots, r$, let $u_j$ be a $C^2$-function defined on $X\backslash P_j$, where $P_j$ is a finite subset of $X$ which may be empty and $u_j(x)$ diverges to infinity as $x\in X\backslash P_j$ approaches $P_j$. We set $P\coloneqq \bigcup_{j=1}^rP_j$. Let $c_{i,j}:X\backslash P\rightarrow \R\ $ be a continuous bounded function for each $i,j=1,\dots, r$. Let $f_j: X\backslash P_j\rightarrow \R$ be continuous nonpositive bounded function for each $j=1,\dots, r$. Finally, let $Y$ be a continuous vector field on $X$. Suppose that $u_1,\dots, u_r$ satisfy the following elliptic system: 
\begin{align*}
\Delta_gu_i+\langle Y, du_j\rangle+\sum_{j=1}^rc_{i,j}u_j=f_i,
\end{align*}
where we denote by $\langle\cdot,\cdot \rangle$ the natural coupling. Suppose also that $(c_{i,j})_{i,j=1,\dots, r}$ satisfies the following:
\begin{enumerate}
\item For each $i,j=1,\dots, r$ such that $i\neq j$, it holds that 
\begin{align*}
c_{i,j}\geq 0.
\end{align*}
\item For each $j=1,\dots, r$, it holds that
\begin{align*}
\sum_{i=1}^rc_{i,j}\leq 0.
\end{align*}
\item For each subsets $A,B\subseteq \{1,\dots, r\}$ satsfying $A\neq\emptyset$, $B\neq\emptyset$, and $A\cup B=\{1,\dots, r\}$, there exist $i\in A$, $j\in B$, and $x\in X$ such that $c_{i,j}(x)>0$.
\end{enumerate}
We consider the following conditions:
\begin{enumerate}[(a)]
\item \label{a} There exists a $j\in \{1,\dots, r\}$ and $x\in X$ such that $f_j(x)<0$.
\item \label{b} $P$ is non-empty.
\item \label{c} It holds that $\sum_{j=1}^ru_j\geq 0$.
\end{enumerate}
Then either condition (\ref{a}) and (\ref{b}) implies that $u_i>0$ for all $i=1,\dots, r$. Condition (\ref{c}) implies that $u_i>0$ for all $i=1,\dots, r$ or $u_1,\dots, u_r$ are all identically $0$. 
}
\end{theo}
We refer the reader to \cite[Appendix A]{STo1} for a generalization of the above. As one of the applications of the above Theorem \ref{DLM}, Dai-Li \cite[Proposition 4.1]{DL1} established the following result:
\begin{theo}[\cite{DL1}]\label{monotone}{\it Let $q$ be a non-zero holomorphic $r$-differential. For each $t>0$, let $h_t$ be the unique diagonal harmonic metric on the cyclic Higgs bundle $(\K_r,\Phi(tq))\rightarrow X$ satisfying $\det(h_t)=1$. We denote by $H_{t,1},\dots, H_{t,r}$ the corresponding Hermitian metrics on $K_X^{-1}\rightarrow X$ constructed as in Section \ref{2.1}. Then for each $t>t^\prime$ and for each $j=1,\dots, r$, we have $H_{t,j}>H_{t^\prime,j}$ at every point of $X$. As a consequence, the energy density $e(h_t)$ of the harmonic map monotonically increases with increasing $t$ at each point of $X$.} 
\end{theo}
Note that in \cite[Proposition 4.1]{DL1}, Dai-Li established the above theorem for more general cyclic type Higgs bundles and the above Theorem \ref{monotone} is a special case treated in \cite[Proposition 4.1]{DL1}. We refer the reader to \cite{DL1, STo1} for further applications of Theorem \ref{DLM}.

\subsection{Dai-Li's characterization for bounded $r$-differentials}\label{2.4}
This subsection reviews the results established in \cite{DL3} (see also \cite{BH1, LT1, Wan1} for the lower rank cases), especially those related to cyclic Higgs bundles. 
The symbols used will mainly be those used in this paper. Let $X$ be an open hyperbolic Riemann surface and $K_X\rightarrow X$ the canonical bundle. Let $g_X$ be the unique complete K\"ahler metric of constant Gaussian curvature $-1$. We denote by $h_X$ the Hermitian metric on $K_X\rightarrow X$ induced by $g_X$, and by $\omega_X$ the K\"ahler form. Let $q$ be a holomorphic $r$-differential and $h=(h_1,\dots, h_r)$ the complete harmonic metric of $\det(h)=1$ on $(\K_r,\Phi(q))\rightarrow X$. Let $f:\widetilde{X}\rightarrow \SL(r,\R)/\SO(r)$ be the corresponding harmonic map from the universal covering space of $X$. We denote by $H_j=h_j^{-1}\otimes h_{j+1}\ j=1,\dots, r-1$ the Hermitian metrics on $K_X^{-1}\rightarrow X$ induced by the harmonic metric $h$, and by $H_0=H_r=h_r^{-1}\otimes h_1\otimes h_q$ the degenerate Hermitian metric on $K_X^{-1}$ induced by the harmonic metric $h$ and the $r$-differential $q$. We set $n\coloneqq [r/2]$. Dai-Li established the following result in \cite[Theorem 4.8]{DL3}:
\begin{theo}[\cite{DL3}]\label{BD}{\it Let $r\geq 3$. The following are equivalent:
\begin{enumerate}
\item $|q|_{h_X}$ is a bounded function.
\item For each $i=1,\dots, n$, $h_X^{\frac{r-(2j-1)}{2}}\otimes h_i^{-1}$ is a bounded function.
\item There exists a constant $C>0$ such that $\vol(H_j)\leq C\omega_X$ for all $j=1,\dots, n$.
\item There exists a constant $C>0$ such that $\sum_{j=1}^r\vol(H_j)\leq C\omega_X$.
\item The Gaussian curvature of the K\"ahler metric induced by $\sum_{j=1}^rH_j$ is bounded above by a negative constant.
\item The sectional curvature of the image of the harmonic map $f$ is bounded above by a negative constant. 
\item The Gaussian curvature of the K\"ahler metric induced by $H_n$ is bounded above by a negative constant.
\item There exist $j\in \{1,\dots, n\}$ and a positive constant $\delta$ such that $H_{j-1}\otimes H_j^{-1}<1-\delta$.
\item There exists a positive constant $\delta$ such that $H_{j-1}\otimes H_j^{-1}<1-\delta$ for all $j=1,\dots, n$.
\item For $r\geq 4$, $r\neq 5$, the curvature of $H_{n-1}$ is bounded above by a negative constant.
\end{enumerate}
}
\end{theo}
For the statement when $r=2$, we refer the reader to \cite[Propositin 4.2]{DL3} and \cite{Wan1}. Furthermore, note that \cite{DL3} also considers harmonic metrics on subcyclic Higgs bundles.
For the proof of the above Theorem \ref{BD}, Dai-Li established the following key lemma in \cite[Lemma 2.2]{DL3}:
\begin{lemm}[\cite{DL3}]\label{delta}{\it Let $g$ be a Riemannian metric on $X$ which is equivalent to $g_X$, i.e., there exists a positive constant $C_1$ such that $C_1^{-1}g_X\leq g\leq C_1g_X$. Suppose that there exists a positive constant $H$ such that the Gaussian curvature $K_g$ of $g$ satisfies 
\begin{align*}
-H\leq K_g\leq 0.
\end{align*} 
Let $u$ be a smooth function that satisfies $-C_2^{-1} K_g\leq u\leq -C_2K_g$ with some positive constant $C_2$ and the following for some positive constant $c$:
\begin{align*}
\Delta_gu\leq cu.
\end{align*}
Then there exists a positive constant $\delta$ that depends only on $C_2, H$, and $c$ such that
\begin{align*}
&u\geq \delta \ \text{and} \\
&K_g\geq -C_2^{-1}\delta.
\end{align*}
}
\end{lemm}
Note that, as can be easily seen from the proof, the assumption in Lemma \ref{delta} that $u$ is smooth can be weakened to the assumption that $u$ belongs to $W^{1,2}_{loc}$. The proof of Lemma \ref{delta} was based on the proof of Wan's classical result \cite[Theorem 13]{Wan1}. 
\subsection{More general subharmonic weight functions}\label{2.3}
For each $q\in H^0(K_X^r)$, the metric $h_q$ induces a singular metric (cf. \cite{GZ1}) $h_q^{-1/r}$ that diverges at the zeros of $q$. We can define a curvature for $h_q^{-1/r}$ in the sense of currents, which is semipositive and has support at the zero points of $q$. We fix a smooth metric $h_\refe$ and denote by $e^{-\phi_q}h_\refe$ the metric $h_q^{-1/r}$, where the weight function $\phi_q$ is defined as $\phi_q\coloneqq \frac{1}{r}\log|q|_{h_\refe}^2$. If we choose a local flat reference metric, then the weight function defines a local subharmonic function. Also, although we will not go into details here, there is a way to identify the weight function with a global psh function on the total space of the dual line bundle excluding zero points, see \cite[Section 2.1]{BB1}. In the sense described above, we call the function $\phi_q$ a {\it subharmonic weight function}, although this is a bit of an abuse of terminology. We consider a more general singular metric $e^{-\varphi}h_\refe$ with semipositive curvature, where the weight function $\varphi$ is a function that is locally a sum of a subharmonic function and a smooth function. For each $q_N\in H^0(K_X^N)$, let $\phi_{q_N}$ be defined as $\frac{1}{N}\log|q_N|_{h_\refe}^2$. Then $e^{-\phi_{q_N}}h_\refe$ is a singular metric with semipositive curvature. Moreover, any semipositive singular metric can be approximated, at least in the $L^1_{loc}$ sense, by a sequence $(e^{-\phi_{q_N}}h_\refe)_{N\in\N}$, where $q_N\in H^0(K_X^N)$ (cf. \cite{GZ1}). In \cite{Miy2, Miy3}, for each semipositive singular metric $e^{-\varphi}h_\refe$ and each $r\geq 2$, the following equation for a diagonal metric $h=(h_1,\dots, h_r)$ on $\K_r\rightarrow X$ was introduced:
\begin{align}
\inum F_{h_j} + \vol(H_{j-1}) - \vol(H_j) = 0 \ \text{for $j=1,\dots, r-1$}, \label{phi}
\end{align}
where $H_1,\dots, H_{r-1}$ are defined as $H_j\coloneqq h_j^{-1}\otimes h_{j+1}$ for each $j=1,\dots, r-1$, and $H_r=H_0$ is defined as follows:
\begin{align*}
H_r = H_0 &\coloneqq h_r^{-1}\otimes h_1 \otimes (e^{-\varphi}h_\refe)^{-r} \\ 
&= H_1^{-1}\otimes \cdots \otimes H_{r-1}^{-1} \otimes (e^{-\varphi}h_\refe)^{-r}.
\end{align*}
For the case where $\varphi=\phi_q=\frac{1}{r}\log|q|_{h_\refe}^2$, equation (\ref{phi}) becomes the Hitchin equation for the cyclic Higgs bundle $(\K_r,\Phi(q))$. For the case where $\varphi=\phi_{q_N}=\frac{1}{N}\log|q_N|_{h_\refe}^2$, equation (\ref{phi}) gives a harmonic metric on a ramified covering space of $X$ (see \cite[Section 2]{Miy2}). As mentioned above, more general $\varphi$ can be approximated by a sequence $(\phi_{q_N})_{N\in\N}$ at least in the $L^1_{loc}$-sense (cf. \cite{GZ1}), therefore equation (\ref{phi}) can be considered to be an equation obtained as the limit when the covering degree increases infinitely or the number of zeros increases infinitely. By making a variable change from $h=(h_1,\dots, h_r)$ to $H_1,\dots, H_{r-1}$, the equation becomes the following:
\begin{align}
\inum F_{H_j} + 2\vol(H_j) - \vol(H_{j-1}) - \vol(H_{j+1}) = 0 \ \text{for $j=1,\dots, r-1$.}\label{cyc2}
\end{align}
We remark on the following two points:
\begin{itemize}
\item If $e^{-\varphi}h_\refe$ is flat, then $H_1=\cdots=H_{r-1}=(e^{-\varphi}h_\refe)^{-1}$ is a solution to PDE (\ref{cyc2}).
\item Suppose that $\varphi=-\infty$. Then equation (\ref{cyc2}) is the same as equation (\ref{cyc}) for the case where $q=0$, and thus for a solution $H$ to equation (\ref{H}), $H_j=\lambda_j H\ (j=1,\dots, r-1)$ is a solution to PDE (\ref{cyc2}), where $\lambda_1,\dots, \lambda_{r-1}$ are constants defined in (\ref{lambda}).
\end{itemize}
Concerning the complete solutions to equation (\ref{phi}), the following holds:
\begin{theo}[\cite{Miy5}]\label{c1} {\it Suppose that $\varphi$ is not identically $-\infty$ unless $X$ is hyperbolic. Suppose also that $\varphi$ satisfies the following assumption:
\begin{enumerate}[$(\ast)$]
\item There exists a compact subset $K\subseteq X$ such that on $X\backslash K$, $e^{\varphi}$ is of class $C^2$.
\end{enumerate}
Then for any two complete solutions $h=(h_1,\dots, h_r)$ and $h^\prime=(h_1^\prime,\dots, h^\prime)$ associated with $\varphi$ satisfying $\det(h)=\det(h^\prime)=1$, we have $h=h^\prime$. 
}
\end{theo}

\begin{theo}[\cite{Miy5}]\label{c2} {\it On the unit disc $\D\coloneqq \{z\in\C\mid |z|<1\}$, for any $\varphi$, there exists a real complete solution $h=(h_1,\dots, h_r)$ associated with $\varphi$. For each $j=1,\dots, [r/2]$, $h_j$ is of class $C^{2j-1,\alpha}$ for any $\alpha\in(0,1)$. }
\end{theo}
\begin{theo}[\cite{Miy5}]\label{c3}{\it
Let $(e^{-\varphi_\epsilon}h_\refe)_{\epsilon>0}$ be a family of smooth semipositive Hermitian metrics on the canonical bundle, each of which is defined on a disc $\D_\epsilon \coloneqq \{z \in \C \mid |z| < 1 - \epsilon\}$, that satisfies the following property:
\begin{itemize}
\item For each $\epsilon>\epsilon^\prime>0$, $\varphi_{\epsilon^\prime}\leq \varphi_\epsilon$ and $(\varphi_\epsilon)_{0<\epsilon<1}$ converges to a function $\varphi$ on $\D$ that is locally a sum of a smooth function and a subharmonic function as $\epsilon\searrow 0$.
\end{itemize}
Then the corresponding family of smooth complete solutions $(h_\epsilon=(h_{1,\epsilon},\dots, h_{r,\epsilon}))_{\epsilon>0}$ monotonically converges to a complete solution $h=(h_1,\dots, h_r)$ associated with $\varphi$ as $\epsilon\searrow 0$.}
\end{theo}
The assumption $(\ast)$ in the uniqueness part of Theorem \ref{c1} is for using the maximum principle on open Riemann surfaces (cf. \cite[Section 3]{LM1}). The author expects that this assumption is unnecessary, but has not yet been able to prove the uniqueness without it. If we can drop the assumption $(\ast)$ in Theorem \ref{main theorem 1}, then, it follows from Theorem \ref{main theorem 2} that for every $\varphi$ on a hyperbolic surface, by lifting to the universal covering space and applying the uniqueness argument on the unit disc (cf. \cite[Proposition 5.7]{LM1}), there exists a complete solution associated with $\varphi$.

Theorem \ref{c3} on approximation for complete solutions is intended to approximately apply the various maximum principle techniques to complete solutions to which these techniques cannot be applied directly. Note that the approximation $(\varphi_\epsilon)_{0<\epsilon<1}$ of $\varphi$ in Theorem \ref{c3} always exists by the standard theory of mollification of subharmonic functions (cf. \cite{Ran1}). If Theorem \ref{c1} on the uniqueness of complete solutions is fully established, then for any $\varphi$ on a hyperbolic Riemann surface, the unique complete solution associated with $\varphi$ can be approximated by a smooth complete solution when lifted to the universal covering space. Similar results on parabolic Riemann surfaces have not yet been established but will be addressed in a subsequent paper. For the purpose of citation when stating the main theorem, we summarize the properties of approximations that we would like a complete solution to satisfy:

\begin{cond}\label{three} Let $h$ be a complete solution to equation (\ref{phi}) associated with $\varphi$. Let $\widetilde{X}$ denote the universal covering space of $X$. We choose the reference metric $h_\refe$ to be flat when pulled back to $\widetilde{X}$. Let $\widetilde{h}$ and $e^{-\widetilde{\varphi}}\widetilde{h}_\refe$ denote the pullbacks of $h$ and $e^{-\varphi}h_\refe$ to $\widetilde{X}$, respectively. Let $(\widetilde{\varphi}_\epsilon)_{\epsilon>0}$ be a mollification of $\widetilde{\varphi}$. We consider the following condition:
\begin{itemize}
\item There exists a family of complete solutions $(h_\epsilon)_{\epsilon>0}$, where each $h_\epsilon$ is associated with $\widetilde{\varphi}_\epsilon$, such that for every $x \in \widetilde{X}$, $h_\epsilon(x) \to \widetilde{h}(x)$ as $\epsilon \searrow 0$.
\end{itemize}
\end{cond}

\subsection{Canonical ensemble}\label{2.5}
We provide a very short review of the notions that appear in relation with the canonical ensemble in the theory of equilibrium statistical mechanics (cf. \cite{DZ1, Gal1, LL1}). Throughout this section, we do not care the so called Boltzmann constant. Let $\Omega$ be a set of microstates. We suppose that $\Omega$ is a finite set for simplicity. Let $E:\Omega\rightarrow \R$ be a map that associates each microstate $i\in\Omega$ with the energy $E_i$ of that state $i$. Let $\beta$ be a positive real constant, called {\it inverse temperature}. Then we call the following probability distribution $(p_i(\beta))_{i\in\Omega}$ on $\Omega$ {\it canonical ensemble}:
\begin{align}
p_i(\beta)\coloneqq \frac{e^{-\beta E_i}}{\sum_{j\in\Omega}e^{-\beta E_j}}\ \text{for $i\in\Omega$}.
\end{align}
We denote by $Z(\beta)$ the normalization constant of the above probability distribution
\begin{align*}
Z(\beta)\coloneqq \sum_{j\in\Omega} e^{-\beta E_j}
\end{align*}
and call it the {\it partition function}. Then logarithm of the partition function multiplied by $-\beta^{-1}$ is called the {\it free energy}:
\begin{align*}
F(\beta)\coloneqq -\frac{1}{\beta}\log Z(\beta).
\end{align*}
We set a new variable $T$ as $T\coloneqq \beta^{-1}$. Although it is a bit of an abuse of notation, let $F(T)$ be the free energy considered as a function of $T$. Then we can check that $F(T)$ is a convex function. At the end of this section, we define the following function, which we call entropy:
\begin{align*}
S(\beta)\coloneqq \beta(\langle E\rangle^{\can}-F(\beta)),
\end{align*}
where $\langle E\rangle^{\can}=\sum_{j\in\Omega}E_jp_j(\beta)$ is the average of $E$ with respect to the canonical ensemble. One can check that the above entropy $S(\beta)$ coincides with the following {\it Shannon entropy}:
\begin{align*}
S(\beta)=-\sum_{j\in \Omega}p_j(\beta)\log p_j(\beta).
\end{align*}
It can be checked that we have the following estimate:
\begin{align*}
0\leq S(\beta)\leq \log |\Omega|,
\end{align*}
where we denote by $|\Omega|$ the cardinality of $\Omega$.
The minimum is attained if and only if $p_j(\beta)=1$ for some $j\in\Omega$ and the others are 0, and the maximum is attained if and only if $p_j(\beta)=1/|\Omega|$ for all $j\in\Omega$, i.e., $E:\Omega\rightarrow \R$ is a constant function. 

\subsection{Shannon entropy for harmonic metrics on cyclic Higgs bundles}\label{2.6}
This subsection reviews on the results established in \cite{Miy4}. Let $h=(h_1,\dots, h_r)$ be a solution to equation (\ref{phi}) and let $H_1,\dots, H_r$ be the Hermitian metrics constructed from $h$ and $e^{-\varphi}h_\refe$ as in Section \ref{2.2}. Then we define a function, which we call entropy, as follows:
\begin{defi}\label{entropy2} 
For each $j=0, 1,\dots, r-1$ and non-zero real number $\beta$, let $p_j(r, \beta, \varphi):X\rightarrow [0,1]$ be a nonnegative function defined as follows:
\begin{align*}
p_j(r, \beta, \varphi)\coloneqq \frac{\vol(H_j)^\beta}{\sum_{j=0}^{r-1}\vol(H_j)^\beta}, 
\end{align*}
where $\vol(H_j)^\beta/\sum_{j=0}^{r-1}\vol(H_j)^\beta$ is understood to be $(\vol(H_j)/\vol(H_\refe))^\beta/\sum_{j=0}^{r-1}(\vol(H_j)/\vol(H_\refe))^\beta$ for some reference metric $H_\refe$, which does not affect $p_j(r, \beta,\varphi)$. We call the following function entropy:
\begin{align*}
S(r,\beta, \varphi)\coloneqq -\sum_{j=0}^{r-1} p_j(r, \beta, \varphi)\log p_j(r, \beta, \varphi).
\end{align*}
\end{defi}
In \cite{Miy4}, the following theorems were established:
\begin{theo}[\cite{Miy4}]\label{suniform}
{\it Let $e^{-\varphi}h_\refe$ be a semipositive singular Hermitian metric on $K_X\rightarrow X$ that is not identically $\infty$ and is non-flat. Suppose that there exists a complete solution $h=(h_1,\dots, h_r)$ to equation (\ref{phi}) in Section \ref{2.3} associated with $\varphi$ that satisfies Condition \ref{three} in Section \ref{2.3}. When $r=2,3$, suppose in addition that $e^{\varphi}h_\refe^{-1}$ belongs to $W^{1,2}_{loc}$.
Then, for any non-zero real number $\beta$, the entropy $S(r,\varphi, \beta)$ constructed from the complete solution $h$ satisfies the following uniform estimate:
\begin{align}
S_{r,\beta}\leq S(r,\varphi,\beta)(x)<\log r\ \text{for any $x\in X$}, \label{SrbSrp}
\end{align}
where $S_{r,\beta}$ is the entropy for the weight function which is identically $-\infty$. Moreover, the equality in the lower bound of $S(r,\varphi,\beta)(x)$ is achieved if and only if $r=2,3$ and $\varphi(x)=-\infty$. 
}
\end{theo}
\begin{theo}[\cite{Miy4}]{\it The following holds:
\begin{align*}
\lim_{r\to\infty} (S_{r,\beta}-\log r)=
\begin{cases}
\frac{2\beta d_\beta}{c_\beta}-\log(c_\beta) \ &\text{if $\beta>-1$} \\
-\infty \ &\text{if $\beta\leq -1$},
\end{cases}
\end{align*}
where $c_\beta$ and $d_\beta$ are defined as follows:
\begin{align*}
&c_\beta\coloneqq \int_0^1s^\beta(1-s)^\beta ds, \\
&d_\beta\coloneqq \int_0^1s^\beta(1-s)^\beta \log s\ ds.
\end{align*}
}
\end{theo}
The introduction of entropy was motivated by the following uniform estimate for harmonic metrics on cyclic Higgs bundles, established in \cite{DL1, DL2, LM1}:
\begin{theo}[\cite{DL1,DL2, LM1}]\label{estimate} 
{\it Let $q$ be a holomorphic $r $-differential with $r\geq 2$, which is non-zero and has at least one zero. Let $h$ be the diagonal complete harmonic metric on the cyclic Higgs bundle $(\K_r, \Phi(q)) \to X$ satisfying $\det(h) = 1$. Then the following estimate holds:
\begin{align}
&\frac{\lambda_{j-1}}{\lambda_j} = \frac{(j-1)(r-j+1)}{j(r-j)} < H_{j-1} \otimes H_j^{-1} < 1 \quad \text{for all } 2 \leq j \leq [r/2], \label{r/2} \\
& H_r \otimes H_1^{-1} < 1.
\end{align}
}
\end{theo}
The assumption in Theorem \ref{suniform} that $e^\varphi h_\refe^{-1}$ belongs to $W^{1,2}_{loc}$ is made in order to apply the mean value inequality in the proof of \cite[Proposition 22]{Miy4}, which is an extension of Theorem \ref{estimate} to more general subharmonic weight functions. This assumption seems unnecessary, but the proof is not yet complete without it. It is used in \cite[Proposition 22]{Miy4} to assert that equality is never achieved in the inequality $H_0 \otimes H_1^{-1} \leq 1$. Therefore, if $r\geq 4$, estimate (\ref{SrbSrp}) holds even without this assumption.

\section{The Omori-Yau maximum principle, the Cheng-Yau maximum principle, and the mean value inequality}\label{CYMMVI}
This subsection briefly reviews the Omori-Yau maximum principle \cite{Omo1, Yau1} (see also \cite[Lemma 3.1]{LM1}), the Cheng-Yau maximum principle \cite{CY1} (see also \cite[Lemma 3.3]{LM1}), and the mean value inequality \cite[Lemma 2.5]{CT1} (see also \cite[Section 2]{DL3}). The following is the Omori-Yau maximum principle \cite{Omo1, Yau1}:
\begin{theo}[\cite{Omo1, Yau1}]{\it Let $(M, g_M)$ be a complete Riemannian manifold with Ricci curvature bounded from below. Then for each real-valued $C^2$-function $u$ on $M$ that is bounded above, there exist an integer $m_0 \in \Z_{\geq 1}$ and a sequence of points $(x_m)_{m \geq m_0}$ in $M$ such that for each $m \geq m_0$,
\begin{align*}
&u(x_m) \geq \sup_M u - \frac{1}{m}, \\
&|du|_{g_M}(x_m) \leq \frac{1}{m}, \\
&(\Delta_{g_M} u)(x_m) \leq \frac{1}{m},
\end{align*}
where $|du|_{g_M}$ denotes the norm of the exterior derivative of $u$, and where $\Delta_{g_M} := -d^{\ast}d$ denotes the negative Laplacian with respect to $g_M$.
}
\end{theo}
The following theorem is known as the Cheng-Yau maximum principle:
\begin{theo}[\cite{CY1}]\label{C-Ym}{\it Let $(M, g_M)$ be a complete Riemannian manifold with Ricci curvature bounded from below. Let $u$ be a real-valued $C^2$ function on $M$ satisfying $\Delta_{g_M} u \geq f(u)$, where $f:\R\rightarrow \R$ is a function. Suppose that there exists a continuous positive function $g:[a,\infty)\rightarrow \R_{>0}$ such that
\begin{enumerate}[(i)]
\item $g$ is non-decreasing;
\item $\liminf_{t\to\infty}\frac{f(t)}{g(t)}>0$;
\item $\int_a^\infty(\int_b^tg(\tau)d\tau)^{-1/2}dt<\infty$ for some $b\geq a$.
\end{enumerate}
Then the function $u$ is bounded above. Moreover, if $f$ is lower semicontinuous, then $f(\sup_M u)\leq 0$.
}
\end{theo}
The following is the Cheng-Yau maximum principle for functions on manifolds with boundary established in \cite[Lemma 3.4]{LM1}:
\begin{theo}[\cite{LM1}]\label{CYMB}
{\it Let $(M, g_M)$ be a complete Riemannian manifold with smooth compact boundary $\partial M$. Suppose that the Ricci curvature of $g_M$ is bounded from below. Let $u$ be a real-valued $C^2$-function on $M$ satisfying $\Delta_{g_M} u \geq f(u)$, where $f:\R\rightarrow \R$ is a function. Suppose that there exists a continuous positive function $g:[a,\infty)\rightarrow \R_{>0}$ such that
\begin{enumerate}[(i)]
\item $g$ is non-decreasing;
\item $\liminf_{t\to\infty}\frac{f(t)}{g(t)} > 0$;
\item $\int_a^\infty \left( \int_b^t g(\tau)\, d\tau \right)^{-1/2} dt < \infty$ for some $b \geq a$.
\end{enumerate}
Then the function $u$ is bounded above. Moreover, if $f$ is lower semicontinuous, then one of the following holds:
\begin{itemize}
\item $\sup_M u = \max_{\partial M} u$;
\item $f(\sup_M u) \leq 0$.
\end{itemize}
}
\end{theo}
Note that in Theorem \ref{CYMB}, we have in mind non-compact manifolds with compact boundary, such as those obtained by removing a relatively compact open subset from a non-compact manifold. The following theorem is known as the mean value inequality (see also \cite[Section 2]{DL3}):
\begin{theo}[\cite{CT1}]\label{MVI}{\it Let $(M,g_M)$ be a complete Riemannian manifold of dimension $n$. We consider the following differential inequality 
\begin{align}
\Delta_{g_M}u\leq c u \label{cuagain}
\end{align}
with some positive constant $c$. Let $x_0\in M$ be a point of $M$. We denote by $B_g(x_0,r)$ the open geodesic ball of radius $r$ centered at $x_0$, and by $\Vol(B(x_0, r))$ the volume of the geodesic ball. We fix a positive constant $R_0$ and suppose the following holds for $x_0$ and $R_0$:
\begin{enumerate}[(i)]
\item The Poincar\'e and the Sobolev inequalities hold for functions supported on $B_g(x_0, R_0)$ with constants $c_p$ and $c_s$;
\item There exists a positive constant $c_2$ such that $\Vol(B(x_0,r))\leq c_2r^n$ for all $r\leq R_0$.
\end{enumerate}
Then there exist positive constants $p_0$ and $C$ depending only on $n, c, c_2, c_p, c_s$ such that for any nonnegative $W^{1,2}$-function $u$ satisfying (\ref{cuagain}) on $B_g(x_0, R_0)$ and any $0<p<p_0$ the following inequality holds:
\begin{align*}
\inf_{x\in B_g(x_0,R_0/4)} u(x)\geq C\left(\int_{B_g(x_0,R_0/4)}u^pd\mu_g\right)^{1/p},
\end{align*}
where $d\mu_g$ is the volume measure. In particular, there exist constants $C>0$ and $0<p<1$ such that 
\begin{align*}
u(x_0)\geq C\left(\int_{B_g(x_0,R_0/4)}u^pd\mu_g\right)^{1/p}.
\end{align*}
}
\end{theo}
In \cite{CT1}, $u$ is assumed to be a $W^{1,2}$-function, but from reading the proof, we see that belonging to $W^{1,2}_{loc}$ is sufficient.
\section{Free energy}\label{3}
In this section we introduce the notion of {\it free energy}, in analogy with the theory of canonical ensemble in statistical physics. Let $H_1,\dots, H_r$ be Hermitian metrics obtained from a solution to equation (\ref{phi}) and $\beta$ a non-zero real constant. 
\begin{defi}{\it We call the following function free energy:
\begin{align*}
F(r,\beta,\varphi, H_\refe)\coloneqq -\frac{1}{\beta}\log(\sum_{j=1}^r(\vol(H_j)/\vol(H_\refe))^\beta),
\end{align*}
where $H_\refe$ is a reference metric. 
}
\end{defi}
Note that the difference between free energy functions with the same reference metric does not depend on the choice of the reference metric.
\begin{rem} The analogy of the canonical ensemble can be further explored to introduce concepts such as energy and heat capacity, but since there are no meaningful results at this stage, we will not introduce them in this paper yet.
\end{rem}
We extend the inequality of Simpson \cite[the first inequality of the proof of Lemma 10.1]{Sim1} (see also \cite{Sim2} and \cite[Section 3]{LM1}) to the free energy function as follows:
\begin{prop}\label{FEineq}{\it Suppose that $\beta$ is positive. Then the following inequality holds in the weak sense:
\begin{align*}
\inum\partial\bar{\partial} F(\beta, r, \varphi, H_\refe)\leq -\frac{\sum_{j=1}^r(\vol(H_{j-1})-\vol(H_j))(\vol(H_{j-1})^\beta-\vol(H_j)^\beta)}{\sum_{j=1}^r\vol(H_j)^\beta}-\inum F_{H_\refe}
\end{align*}
}
\end{prop}
\begin{proof} Let $u_1,\dots, u_r$ be the canonical basis of $\R^r$. Let $f_1,\dots, f_r$ be real valued smooth functions on $X$. We define a vector valued function $b$ as follows:
\begin{align*}
b\coloneqq \sum_{j=1}^re^{f_j/2}u_j.
\end{align*}
Then the following holds:
\begin{align*}
\inum\partial\bar{\partial}\log|b|^2&=\frac{2\inum}{|b|^2}\langle \partial \bar{\partial} b,b\rangle+\frac{2\inum}{|b|^2}\langle\partial b\wedge\bar{\partial}b\rangle-\frac{4\inum}{|b|^4}\langle\partial b,b\rangle \wedge\langle\bar{\partial}b,b\rangle. \\
&\geq \frac{2\inum}{|b|^2}(\langle\partial \bar{\partial}b, b\rangle-\langle\partial b\wedge\bar{\partial}b\rangle).
\end{align*}
The Laplacian of $b$ is calculated as follows:
\begin{align*}
\inum \partial \bar{\partial}b&=\inum \partial \bar{\partial} \sum_{j=1}^re^{f_j/2}u_j \\
&=\inum \sum_{j=1}^r\partial(f_j/2)\wedge\bar{\partial}(f_j/2)e^{f_j/2}u_j+\inum \sum_{j=1}^r\partial \bar{\partial} (f_j/2)e^{f_j/2} u_j.
\end{align*}
Therefore we have
\begin{align}
\inum\partial\bar{\partial}\log|b|^2
&\geq \frac{1}{|b|^2}(2\inum\langle\partial \bar{\partial}b, b\rangle-2\inum \langle\partial b\wedge\bar{\partial}b\rangle) \\
&\geq \frac{1}{|b|^2}\left(\frac{\inum}{2}\sum_{j=1}^r\left(\partial f_j\wedge\bar{\partial}f_j\right)e^{f_j}+\inum \sum_{j=1}^r(\partial \bar{\partial}f_j)e^{f_j}-2\inum \langle\partial b\wedge\bar{\partial}b\rangle\right) \\
&=\frac{1}{|b|^2}\left(2\inum\langle \partial b\wedge\bar{\partial}b\rangle +\inum\sum_{j=1}^r\partial \bar{\partial}f_je^{f_j}-2\inum \langle\partial b\wedge\bar{\partial}b\rangle\right) \\
&=\frac{\inum}{|b|^2}\sum_{j=1}^r\partial \bar{\partial}f_je^{f_j}. \label{mollification}
\end{align}
By using the mollification argument (cf. \cite{Miy3}), inequality (\ref{mollification}) holds even if $f_1,\dots, f_r$ are not smooth but instead belong to $L^1_{loc}$ and are locally bounded above. For each $j=1,\dots,r$, we set $f_j \coloneqq \beta \log(\vol(H_j)/\vol(H_\refe))$. Then we have $F(\beta,r,\varphi, H_\refe)=-\frac{1}{\beta}\log|b|^2$. We also set $v_j\coloneqq u_{j+1}-u_j$ for each $j=1,\dots, r-1$ and $v_r\coloneqq u_1-u_r$. Then it holds that
\begin{align*}
\inum F_{H_j}+(\sum_{k=1}^r\vol(H_k)v_k,v_j)=0.
\end{align*}
From inequality (\ref{mollification}), we have
\begin{align*}
\inum\partial\bar{\partial}\log|b|^2&\geq \frac{1}{|b|^2}\sum_{j=1}^r(-\beta \inum F_{H_j}+\beta\inum F_{H_\refe})e^{f_j} \\
&=\frac{1}{|b|^2}\sum_{j=1}^r(\beta\sum_{k=1}^r(\vol(H_k)v_k,v_j)+\beta\inum F_{H_\refe})e^{f_j} \\
&=\beta\frac{\langle\sum_{j=1}^r\vol(H_j)v_j, \sum_{k=1}^r\vol(H_k)^\beta v_k\rangle}{\sum_{j=1}^r\vol(H_j)^\beta}+\beta\inum F_{H_\refe} \\
&=\beta \frac{\langle\sum_{j=1}^r(\vol(H_{j-1})-\vol(H_j))u_j, \sum_{k=1}^r(\vol(H_{k-1})^\beta-\vol(H_k)^\beta)u_k \rangle}{\sum_{j=1}^r\vol(H_j)^\beta}+\beta\inum F_{H_\refe} \\
&=\beta \frac{\sum_{j=1}^r(\vol(H_{j-1})-\vol(H_j))(\vol(H_{j-1})^\beta-\vol(H_j)^\beta)}{\sum_{j=1}^r\vol(H_j)^\beta}+\inum \beta F_{H_\refe}.
\end{align*}
This implies the claim.
\end{proof}

\section{Redundancy}\label{4}
In this section we introduce the new notion which we call {\it redundancy function} and {\it lower redundancy constant}, following the redundancy notion in information theory \cite{Sh1}. 
\begin{defi} We call the following function {\it redundancy function}:
\begin{align*}
R(r,\beta,\varphi)\coloneqq 1-S(r,\beta,\varphi)/\log r.
\end{align*}
\end{defi}
\begin{defi} We call the following constant {\it lower redundancy constant}:
\begin{align*}
\underline{R}_{r,\beta,\varphi}\coloneqq \inf_{x\in X}R(r,\beta,\varphi)(x).
\end{align*}
\end{defi}
\begin{rem} In a similar manner, the {\it upper redundancy constant} can be defined as 
\begin{align*}
\overline{R}_{r,\beta,\varphi}\coloneqq \sup_{x\in X}R(r,\beta,\varphi)(x),
\end{align*}
but we will not go into detail in this paper.
\end{rem}
\section{Main theorems}\label{5}
The following are our main theorems:

\begin{theo}\label{main theorem b}{\it Let $e^{-\varphi}h_\refe$ and $e^{-\varphi^\prime}h_\refe$ be semipositive singular Hermitian metrics on $K_X \to X$. Suppose that there exist complete solutions $h$ and $h^\prime$ associated with $\varphi$ and $\varphi^\prime$, respectively, which satisfy Condition \ref{three} in Section \ref{2.3}. Consider the following conditions on $\varphi$ and $\varphi^\prime$:
\begin{enumerate}[(i)]
\item For every $x\in X$, $\varphi(x)\geq \varphi^\prime(x)$, and for at least one $x\in X$, $\varphi(x)>\varphi^\prime(x)$.
\item For every $x\in X$, $\inum\bar{\partial}\partial \varphi(x)\geq \inum\bar{\partial}\partial\varphi^\prime(x)$.
\item When the mollification $(\widetilde{\varphi}_\epsilon)_{\epsilon>0}$ of the pullback of $\varphi$ to the universal covering is taken as in Condition \ref{three}, the metric $e^{\widetilde{\varphi}_\epsilon}\widetilde{h}_\refe^{-1}$ is complete for each $\epsilon>0$.
\end{enumerate}
Suppose $\varphi$ and $\varphi^\prime$ satisfy condition (i), $r=2$ or $3$, and $\beta$ is positive. Then the free energy functions for $h$ and $h^\prime$ satisfy the following inequality:
\begin{align*}
0 < F(r,\beta,\varphi^\prime, H_\refe)(x) - F(r,\beta,\varphi,H_\refe)(x) \ \text{for each $x\in X$}.
\end{align*}
Suppose that $\varphi$ and $\varphi^\prime$ satisfy conditions (i), (ii), and (iii), and that $\beta$ is positive. Then for any $r$, the free energy functions for $h$ and $h^\prime$ satisfy:
\begin{align*}
0 < F(r,\beta,\varphi^\prime, H_\refe)(x) - F(r,\beta,\varphi,H_\refe)(x) < r(\varphi(x) - \varphi^\prime(x)) + \frac{1}{\beta}\log r \ \text{for each $x\in X$}.
\end{align*}
}
\end{theo}

\begin{theo}\label{main theorem a}{\it Let $r$ be either $2$ or $3$. Let $e^{-\varphi}h_\refe$ and $e^{-\varphi^\prime}h_\refe$ be semipositive singular Hermitian metrics on $K_X \to X$. Suppose that there exist complete solutions $h$ and $h^\prime$ associated with $\varphi$ and $\varphi^\prime$, respectively, which satisfy Condition \ref{three} in Section \ref{2.3}. Suppose also that $\varphi$ and $\varphi^\prime$ satisfy the following:
\begin{itemize}
\item For each $x \in X$, $\varphi(x) \geq \varphi^\prime(x)$, and for at least one $x \in X$, $\varphi(x) > \varphi^\prime(x)$.
\item For each $x \in X$, $\inum\bar{\partial}\partial \varphi(x) \geq \inum\bar{\partial}\partial \varphi^\prime(x)$.
\item When the mollification $(\widetilde{\varphi}_\epsilon)_{\epsilon>0}$ of the pullback of $\varphi$ to the universal covering is taken as in Condition \ref{three}, the metric $e^{\widetilde{\varphi}_\epsilon}\widetilde{h}_\refe^{-1}$ is complete for each $\epsilon>0$.
\end{itemize}
Then the entropy functions for $h$ and $h^\prime$ satisfy:
\begin{align}
S(r,\beta,\varphi)(x) \geq S(r,\beta,\varphi^\prime)(x) \quad \text{for all } x \in X. \label{S(x)2}
\end{align}
If, moreover, $\varphi^\prime - \varphi$ belongs to $W^{1,2}_{loc}$ and $\varphi(x) > \varphi^\prime(x)$ holds at least at one point of $X$, then inequality (\ref{S(x)2}) becomes strict for all points of $X$.
}
\end{theo}

\begin{theo}\label{main theorem c}{\it Suppose that $X$ is the unit disc $\D\coloneqq \{z\in\C\mid |z|<1\}$. Let $h_X$ be the Hermitian metric on $K_X\rightarrow X$ induced by the Poincar\'e metric. Let $e^{-\varphi}h_X$ be a semipositive singular Hermitian metric on the canonical bundle satisfying the following condition:
\begin{itemize}
\item There exists a compact subset $K\subseteq X$ such that on $X\backslash K$, $e^{\varphi}$ is of class $C^2$.
\end{itemize}
We consider the entropy function and the free energy function constructed from the unique complete solution associated with $\varphi$. We set a constant $M_\varphi$, which may be infinite, by $M_\varphi\coloneqq \sup_X e^\varphi$. Suppose further that $e^\varphi$ belongs to $W^{1,2}_{loc}$ when $r=2,3$. Then the following are equivalent:
\begin{enumerate}[(a)]
\item $e^\varphi$ is bounded above, i.e., $M_\varphi$ is finite.
\item For each non-zero real constant $\beta$ and for each $r\geq 2$, there exists a positive constant $\delta$ that such that the following uniform estimate holds:
\begin{align*}
S(r,\beta,\varphi)\leq \log r-\delta.
\end{align*}
\item The following lower redundancy constant $\underline{R}_{r,\beta,\varphi}$ is positive:
\begin{align*}
\underline{R}_{r,\beta,\varphi}\coloneqq \inf_{x\in X}(1-(S(r,\beta,\varphi)(x)/\log r)).
\end{align*}
\item For each positive constant $\beta$ and for each $r\geq 2$, there exists a positive constant $C_1$ such that the following uniform estimate holds:
\begin{align*}
F(r,\beta,\varphi, H_\refe)-F(r,\beta,-\infty, H_\refe)\leq C_1.
\end{align*}
\item For each negative constant $\beta$ and for each $r\geq 2$, there exists a positive constant $C_2$ such that the following uniform estimate holds:
\begin{align*}
F(r,\beta,\varphi, H_\refe)-F(r,\beta,-\infty, H_\refe)\geq -C_2.
\end{align*}
\item For each positive constant $\beta$ and for each $r\geq 2$, there exists a positive constant $C_2$ such that the following uniform estimate holds:
\begin{align*}
\inum \partial \bar{\partial} (F(r,\beta,\varphi, H_\refe)-F(r,\beta,-\infty, H_\refe))\leq -C_3 e^{-F(r,\beta,\varphi, H_\refe)}\vol(H_\refe)+\omega_X,
\end{align*}
where $\omega_X$ is the K\"ahler form of $h_X$.
\end{enumerate}
}
\end{theo}

\begin{rem}\label{phiremark} From \cite[Lemma 19 and Corollary 20]{Miy5} and the proof of Theorem \ref{main theorem c}, we see that the constants $\delta, C_1,C_2$ and $C_3$ in Theorem \ref{main theorem 3} can all be chosen to depend only on $M_\varphi, \beta,$ and $r$, provided that $e^\varphi$ is globally of class $C^2$. 
\end{rem}
When introducing a new measure to gauge the singularity of $\varphi$, such as entropy and free energy, the question ``To what extent is the given $\varphi$ singular?'' is the most fundamental issue the author wishes to raise alongside the introduction of these concepts.
\begin{que} How do entropy and free energy vary when a subharmonic function varies in various situations in potential theory and complex analysis?
\end{que}
\begin{que} What is the behavior of entropy and free energy associated with random sections (cf. \cite{BCHM1, SZ1})?
\end{que}
\begin{que} If we fix a holomorphic map from $X$ to $\CP^1$ and pull a complex dynamical system on $\CP^1$ back to $X$, how will the associated entropy and free energy behave?
\end{que}
\begin{que} Regarding Theorem \ref{main theorem b}, what can be said about the converse of the statement? That is, what is the necessary condition on $\varphi$ for the free energy to decrease at each point?
\end{que} 
\begin{que} Regarding Theorem \ref{main theorem a}, what can be said about the converse of the statement? That is, what is the necessary condition on $\varphi$ for the entropy to increase at each point?
\end{que}
\begin{que} Regarding Theorem \ref{main theorem b}, if we weaken the condition that the free energy decreases at each point on the Riemann surface $X$ to the condition that the integral of the free energy with respect to a probability measure $\mu$ on $X$ decreases, what are the necessary and sufficient conditions on $\varphi$ for the free energy to decrease in this sense?
\end{que}
\begin{que} Regarding Theorem \ref{main theorem a}, if we weaken the condition that the entropy increases at each point on the Riemann surface $X$ to the condition that the integral of the entropy with respect to a probability measure $\mu$ on $X$ increases, what are the necessary and sufficient conditions on $\varphi$ for the entropy to increase in this sense?
\end{que}
\begin{que} What more can we say about the lower and upper redundancy constants beyond the theorem above?
\end{que}
\begin{que} Can the lower redundancy constant on the complex plane take a value other than zero?
\end{que}
\begin{que} If we define the notion of {\it energy} $E(r,\beta,\varphi, H_\ast)$ as 
\begin{align*}
E(r,\beta,\varphi, H_\ast)\coloneqq \frac{\partial }{\partial \beta}F(r,\beta,\varphi, H_\refe), 
\end{align*}
then can energy increase or decrease at all points of $X$?
\end{que}
\section{Proof}\label{6}
Throughout this section, for each Hermitian metric $H$ on $K_X^{-1}\rightarrow X$, we denote by $\Lambda_H$ the adjoint of $\omega_H\wedge$, where $\omega_H$ is the K\"ahler form induced by $H$.
\subsection{The completeness of $e^\varphi h_\refe^{-1}$ and its mutual boundedness with complete solutions}
Before giving the proofs of the main theorems, this subsection prepares a lemma that extends \cite[Proposition 3.37]{LM1} to more general subharmonic weight functions. The precise assertion is as follows:
\begin{lemm}\label{phicomp}{\it Let $e^{-\varphi}h_\refe$ is a semipositive singular Hermitian metric on $K_X\rightarrow X$. Suppose that outside a compact subset $K\subseteq X$, $\varphi>-\infty$ and $e^\varphi h_\refe^{-1}$ is of class $C^2$. Let $h=(h_1,\dots, h_r)$ be a complete solution to equation (\ref{phi}) associated with $\varphi$. We denote by $H_0, H_1,\dots, H_{r-1}$ the Hermitian metrics on $K_X^{-1}\rightarrow X$ constructed from $h$ and $e^{-\varphi}h_\refe$ as in Section \ref{2.3}. Then the following are equivalent:
\begin{enumerate}[(i)]
\item $e^\varphi h_\refe^{-1}$ is complete outside $K\subseteq X$.
\item There exists a positive constant $C$ such that for each $j=1,\dots, r-1$, the following holds:
\begin{align*}
H_j\leq Ce^\varphi h_\refe^{-1}.
\end{align*}
\item There exists a positive constant $C$ such that for each $j=1,\dots, r-1$, the following holds:
\begin{align*}
H_j\leq CH_0.
\end{align*}
\end{enumerate}
}
\end{lemm}
\begin{proof} From \cite[Corollary 22]{Miy5}, $H_1,\dots, H_{r-1}$ are mutually bounded. Therefore it is easy to see that (ii) and (iii) are equivalent. Suppose that (i) holds. Let $H$ be a Hermitian metric on $K_X^{-1}\rightarrow X$ such that $H$ is mutually bounded with $e^\varphi h_\refe^{-1}$ outside $K\subseteq X$. Then from \cite[Corollary 20]{Miy5}, we have (ii). Obviously, (ii) implies (i). This completes the proof. 
\end{proof}
\subsection{Proof of Theorem \ref{main theorem b}}
\begin{proof}[Proof of Theorem \ref{main theorem b}] 
Let $h=(h_1,\dots, h_r)$ and $h^\prime=(h_1^\prime,\dots, h_{r-1}^\prime)$ be real complete solutions to equation (\ref{phi}) associated with $\varphi$ and $\varphi^\prime$, respectively. We denote by $H_1,\dots,H_{r-1}, H_r=H_0$ (resp. $H_1^\prime,\dots, H_{r-1}^\prime, H_r^\prime=H_0^\prime$) the Hermitian metrics on $K_X^{-1}\rightarrow X$ constructed from $h$ and $e^{-\varphi}h_\refe$ (resp. $h^\prime$ and $e^{-\varphi^\prime}h_\refe$) as in Section \ref{2.3}. For each $j=0,1,\dots, r$, we set $\sigma_j\coloneqq \log(H_j\otimes H_j^{\prime-1})$ and $c_j\coloneqq \int_0^1\vol(H_j)^t\vol(H_j^\prime)^{1-t}dt$. From the symmetry, we have $\sigma_j = \sigma_{r-j}$ and $c_j=c_{r-j}$ for each $j = 0, \dots, r$. 
The following holds:
\begin{align}
\inum\partial\bar{\partial}\sigma_j&=2\vol(H_j)-\vol(H_{j-1})-\vol(H_{j+1})-2\vol(H_j^\prime)+\vol(H_{j+1}^\prime)+\vol(H_{j-1}^\prime) \notag \\
&=2c_j\sigma_j-c_{j-1}\sigma_{j-1}-c_{j+1}\sigma_{j+1} \ \text{for $j=1,\dots, r-1$}, \label{sigmaj} 
\end{align}
Moreover, if we assume that $\inum\bar{\partial}\partial \varphi \geq \inum\bar{\partial}\partial \varphi^\prime$, then we have
\begin{align}
\inum\partial \bar{\partial}\sigma_0&\leq 2\vol(H_0)-\vol(H_1)-\vol(H_{r-1})-2\vol(H_0^\prime)+\vol(H_1^\prime)+\vol(H_{r-1}^\prime) \notag \\
&=2c_0\sigma_0-c_1\sigma_1-c_{r-1}\sigma_{r-1} \label{sigma0}.
\end{align}
Suppose that $r$ is eigher 2 or 3 and $\varphi\geq \varphi^\prime$. Then from (\ref{sigmaj}) and the property $\sigma_1=\sigma_2$ for the case where $r=3$, we have
\begin{align}
\inum\Lambda_{H_1^\prime}\partial\bar{\partial}\sigma_1&=(4-r)\Lambda_{H_1^\prime}(c_1\sigma_1-c_0\sigma_0) \notag \\
&=(4-r)\Lambda_{H_1^\prime}(c_1\sigma_1+(r-1)c_0\sigma_1-r(\varphi-\varphi^\prime)c_0) \notag \\
&\leq (4-r)\Lambda_{H_1^\prime}(c_1+(r-1)c_0)\sigma_1. \label{4-r}
\end{align}
Suppose that $\varphi$ and $\varphi^\prime$ are smooth. Then from \cite[Corollary 20]{Miy5}, $\sigma_1$ is bounded below. Also, for the same reason as in \cite[Lemma 2.4]{LM1}, the Gaussian curvature of $H_1^\prime$ is bounded from below. Then from (\ref{4-r}) and the Omori-Yau maximum principle, there exist $m_0\in\Z_{\geq 1}$ and a sequence of points $(x_m)_{m\geq m_0}$ on $X$ such that
\begin{align}
-\sigma_1(x_m)&\geq \sup_X(-\sigma_1)-1/m, \label{mmm}\\
1/m&\geq -(4-r)\Lambda_{H_1^\prime}(c_1(x_m)+(r-1)c_0(x_m))\sigma_1(x_m) \label{1mxm}
\end{align}
From \cite[Corollary 20]{Miy5}, the coefficient $\Lambda_{H_1^\prime}c_1$ is bounded from below by a positive constant. Therefore from (\ref{mmm}) and (\ref{1mxm}), by taking the limit as $m\to\infty$, we obtain $\inf_X\sigma_1\geq 0$. 
Suppose that $h$ and $h^\prime$ satisfy Condition \ref{three} in Section \ref{2.3}. Since it is sufficient to lift $X$ to the universal covering space and show the statement of the theorem, we will assume from the beginning that $h$ and $h^\prime$ have an approximation $(h_\epsilon)_{\epsilon>0}$ and $(h_\epsilon^\prime)_{\epsilon>0}$ in Condition \ref{three}. We set $H_{1,\epsilon}\coloneqq h_{1,\epsilon}^{-1}\otimes h_{2,\epsilon}, H_{1,\epsilon}^\prime\coloneqq h_{1,\epsilon}^{\prime-1}\otimes h_{2,\epsilon}^\prime$, and $\sigma_{1,\epsilon}\coloneqq \log(H_{1,\epsilon}\otimes H_{1,\epsilon}^{\prime-1})$. Since we have proved that $\sigma_1\geq 0$ for the case where $\varphi$ and $\varphi^\prime$ are smooth, $\sigma_{1, \epsilon}\geq 0$. By taking the limit $\epsilon\searrow 0$, we have $\sigma_1\geq 0$. Let $c$ be a positive constant defined as follows:
\begin{align}
c\coloneqq (4-r)\sup_X(\Lambda_{H_1}(c_1+(r-1)c_0)). \label{constantc}
\end{align}
Note that $c$ is finite since from \cite[Lemma 19 and Corollary 20]{Miy4}, for each $\epsilon>0$, $\Lambda_{H_{1,\epsilon}}\int_0^1\vol(H_{1,\epsilon})^t\vol(H_{1,\epsilon}^\prime)^{1-t}dt$ and $\Lambda_{H_{1,\epsilon}}\int_0^1\vol(H_{0,\epsilon})^t\vol(H_{0,\epsilon}^\prime)^{1-t}dt$ are bounded above by constants independent of $\epsilon$.
From \eqref{4-r}, by replacing the contraction operator $\Lambda_{H_1^\prime}$ with $\Lambda_{H_1}$, we have
\begin{align*}
\inum\Lambda_{H_1}\partial\bar{\partial}\sigma_1\leq c\sigma_1.
\end{align*}
Let $x_0 \in X$ be an arbitrary point. Then from the mean value inequality theorem, there exists a positive constant $C$ and $0<p<1$ such that
\begin{align}
\sigma_1(x_0)\geq C\left(\int_{B_g(x_0,1/4)}\sigma_1^pd\mu_g\right)^{1/p}, \label{ux0}
\end{align}
where $g$ is the K\"ahler metric induced by $H_1$.
Let $A$ be the set $\{x\in X\mid \sigma_1(x)=0\}$. Inequality (\ref{ux0}) implies that the set $A$ is both closed and open. Suppose that there exists at least one $x\in X$ such that $\varphi(x)>\varphi^\prime(x)$. This implies that $\sigma_1$ is a non-constant function. Therefore, $A$ is an empty set, and thus $\sigma_1>0$. We show that this implies $F(r,\beta,\varphi^\prime, H_\refe)>F(r,\beta,\varphi,H_\refe)$. This follows from the following elementary calculus and \cite[Proposition 22]{Miy4}. For each constant $a$, we define a function $f_a:\R\rightarrow \R$ as $f_a(t)\coloneqq e^t+e^{a-t}$. Then the derivative of $f_a$ can be computed as follows:
\begin{align*}
\frac{df_a}{dt}(t)=e^t-e^{a-t}.
\end{align*} 
Therefore, $f_a$ strictly monotonically increases if $1>e^{a-2t}$. On the other hand, from \cite[Proposition 22]{Miy4}, we have $H_0/H_1<1$ and $H_0^\prime/H_1^\prime<1$. This implies $F(r,\beta,\varphi^\prime, H_\refe)>F(r,\beta,\varphi,H_\refe)$.

We next prove the latter half of the statement. We use the argument of the proof of Dai-Li's maximum principle (see Section \ref{2.2}).
Suppose that $\varphi\geq \varphi^\prime$ and that $\varphi$ and $\varphi^\prime$ are smooth. Then from \cite[Corollary 20]{Miy5}, $\sigma_1,\dots, \sigma_{r-1}$ are bounded below. Furthermore, suppose that $e^{\varphi}h_\refe^{-1}$ is complete outside a compact subset $K\subseteq X$. Then from \cite[Corollary 20]{Miy5} and Lemma \ref{phicomp}, $\sigma_0$ is bounded beow. For each subset $S\subseteq \{0,1,\dots, r-1\}$, we set
\begin{align*}
\sigma_S&\coloneqq \sum_{j\in S}\sigma_j, \\
b_S&\coloneqq \inf_X\sigma_S, \\
\widetilde{b}_S&\coloneqq \underset{j\notin S, k\in S}{\min}\{b_{\{j\}\cup S}, b_{S\setminus \{k\}}\},
\end{align*}
where the sum over the empty set is understood to be zero. Furthermore, for each $i, j \in \{0, 1, \dots, r-1\}$, we set
\begin{align*}
a_{i,j} \coloneqq 
\begin{cases}
-2c_j & \text{if } i = j, \\
c_{j+1} & \text{if } j = i + 1 \ \text{or} \ i=r-1, j=0,\\
c_{j-1} & \text{if } j = i - 1 \ \text{or} \ i=0, j=r-1, \\
0 & \text{otherwise}.
\end{cases}
\end{align*}
Then we have
\begin{align*}
\inum\partial\bar{\partial}\sigma_i+\sum_{j=0}^{r-1}a_{i,j}\sigma_j\leq 0 \ \text{for each $i=1,\dots, r-1$.}
\end{align*}
For each $S\subseteq \{0,1,\dots, r-1\}$, from the definition of $\sigma_S$, we have
\begin{align}
\inum\partial\bar{\partial}\sigma_S+\sum_{i\in S}\sum_{j=0}^{r-1}a_{i,j}\sigma_j\leq 0. \label{S}
\end{align}
The left hand side is calculated as follows:
\begin{align}
\inum\partial\bar{\partial}\sigma_S+\sum_{i\in S}\sum_{j=0}^{r-1}a_{i,j}\sigma_j&=\inum\partial\bar{\partial}\sigma_S+\sum_{i\in S}\sum_{k\in S}a_{i,k}(\sigma_S-\sigma_{S\setminus\{k\}})+\sum_{i\in S}\sum_{j\notin S}a_{i,j}(\sigma_{\{j\}\cup S}-\sigma_S). \label{LHS}
\end{align}
On the other hand, since we have $\sum_{k\in S}a_{i,k}\leq 0$ for each $i\in S$ and $a_{i,j}\geq 0$ for each $i\neq j$, the constant $\widetilde{b}_S$ satisfies
\begin{align}
\sum_{i\in S}\sum_{k\in S}a_{i,k}(\sigma_{S\setminus\{k\}}-\widetilde{b}_S)+\sum_{i\in S}\sum_{j\notin S}a_{i,j}(\widetilde{b}_S-\sigma_{\{j\}\cup S})\leq 0. \label{bS}
\end{align}
From (\ref{LHS}) and (\ref{bS}), we have
\begin{align}
\inum\partial\bar{\partial}\sigma_S+\sum_{i\in S}\sum_{k\in S}a_{i,j}(\sigma_S-\widetilde{b}_S)+\sum_{i\in S}\sum_{j\notin S}a_{i,j}(\widetilde{b}_S-\sigma_S)\leq 0. \label{sigmaSbS}
\end{align}
We choose one from among $H_1^\prime,\dots, H_{[r/2]}^\prime$ and denote it by $H^\prime$. From (\ref{sigmaSbS}), we have
\begin{align*}
\inum\Lambda_{H^\prime}\partial\bar{\partial}(\sigma_S-\widetilde{b}_S)+(\sum_{i\in S}\sum_{k\in S}\Lambda_{H^\prime}a_{i,k}-\sum_{i\in S}\sum_{j\notin S}\Lambda_{H^\prime}a_{i,j})(\sigma_S-\widetilde{b}_S)\leq 0.
\end{align*}
Since $\sigma_S$ is bounded below and the Gaussian curvature of $H^\prime$ is bounded below, by the Omori-Yau maximum principle, there exist $m_0\in\Z_{\geq 1}$ and a sequence of points $(x_m)_{m\geq m_0}$ on $X$ such that
\begin{align}
-\sigma_S(x_m)+\widetilde{b}_S&\geq \sup_X(-\sigma_S+\widetilde{b}_S)-1/m, \label{SxmbS}\\
1/m&\geq (\sum_{i\in S}\sum_{k\in S}\Lambda_{H^\prime}a_{i,k}(x_m)-\sum_{i\in S}\sum_{j\notin S}\Lambda_{H^\prime}a_{i,j}(x_m))(\sigma_S(x_m)-\widetilde{b}_S). \label{aijxm}
\end{align}
We set $f_m\coloneqq (\sum_{i\in S}\sum_{k\in S}\Lambda_{H^\prime}a_{i,k}(x_m)-\sum_{i\in S}\sum_{j\notin S}\Lambda_{H^\prime}a_{i,j}(x_m))$. Suppose that $S$ is nonempty and not equal to $\{0,1,\dots, r-1\}$. Then, by the assumption, there exists $i \in S$ such that $i \in S$ and $i+1 \notin S$, where $r - 1 + 1$ is understood to be $0$. We choose such an $i$. If $i \leq r - 2$, then $f_m$ is bounded from below by a positive constant, since $\Lambda_{H_1'} a_{i,i+1}$ is bounded from below by a positive constant by \cite[Corollary 20]{Miy4}. Suppose that $e^{\varphi} h_{\refe}^{-1}$ is non-degenerate and complete. Then, from Lemma \ref{phicomp}, if $i = r - 1$, $f_m$ is again bounded from below by a positive constant, since $a_{r-1, 0}$ is bounded from below by a positive constant.
Therefore, by taking the limit as $m\to\infty$ in (\ref{SxmbS}) and (\ref{aijxm}), we have $\sigma_S\geq \widetilde{b}_S$ for each proper nonempty subset $S\subseteq \{0,1,\dots, r-1\}$. By using this estimate, we show that $b_S \geq 0$ for each subset $S \subseteq \{0, \dots, r - 1\}$ as follows. 
Let $S_0$ be a subset of $\{0,1,\dots, r-1\}$ such that $b_{S_0}=\min_{S\subseteq \{0,1,\dots, r-1\}}b_S$. If $b_{S_0}\geq 0$, there is nothing more to prove, so we will assume that this is not the case. For $S=\emptyset$, we have by definition $b_S=0$, and thus $S_0\neq \emptyset$. For $S=\{0, \dots, r - 1\}$, from the assumption $\varphi \geq \varphi^\prime$, we have $b_S \geq 0$ and thus we have $S_0\subsetneq \{0,1,\dots, r-1\}$. Therefore, $S_0$ is a proper nonempty subset of $\{0,1,\dots, r-1\}$. From the estimate we have proved, it holds that $\sigma_{S_0}\geq \widetilde{b}_{S_0}\geq b_{S_0}$, and thus we obtain $\widetilde{b}_{S_0}=b_{S_0}$. From (\ref{S}), (\ref{LHS}), and the Omori-Yau maximum principle, there exist $m_0\in\Z_{\geq 1}$ and a sequence of points $(x_m)_{m\geq m_0}$ on $X$ such that
\begin{align}
-\sigma_{S_0}(x_m) \geq &\sup_X(-\sigma_{S_0})-1/m, \label{sS0xm}\\
1/m\geq &\sum_{i\in S_0}\sum_{k\in S_0}\Lambda_{H^\prime} a_{i,k}(x_m)(\sigma_{S_0}(x_m)-\sigma_{S_0\setminus\{k\}}(x_m)) \notag \\
&+\sum_{i\in S_0}\sum_{j\notin S_0}\Lambda_{H^\prime} a_{i,j}(x_m)(\sigma_{\{j\}\cup S_0}(x_m)-\sigma_{S_0}(x_m)). \label{jS0xm}
\end{align}
From (\ref{sS0xm}) and (\ref{jS0xm}), we have
\begin{align}
1/m\geq &\sum_{i\in S_0}\sum_{k\in S_0}\Lambda_{H^\prime} a_{i,k}(x_m)(b_{S_0}-1/m-\sigma_{S_0\setminus\{k\}}(x_m)) \notag \\
&+\sum_{i\in S_0}\sum_{j\notin S_0}\Lambda_{H^\prime} a_{i,j}(x_m)(\sigma_{\{j\}\cup S_0}(x_m)-b_{S_0}+1/m) \\
\geq &\sum_{i\in S_0}\sum_{k\in S_0}\Lambda_{H^\prime} a_{i,k}(x_m)(b_{S_0}-1/m-\sigma_{S_0\setminus\{k\}}(x_m)) \notag \\
\geq &\sum_{i\in S_0}\sum_{k\in S_0}\Lambda_{H^\prime} a_{i,k}(x_m)(b_{S_0}-\sigma_{S_0\setminus\{k\}}(x_m)) \notag \\
\geq &\sum_{i\in S_0}\Lambda_{H^\prime} a_{i,k}(x_m)(b_{S_0}-\sigma_{S_0\setminus\{k\}}(x_m)) \ \text{for each $k\in S$}\label{j} 
\end{align}
Since $S_0$ is a proper nonempty subset, there exists $k\in S_0$ such that $\sum_{i\in S_0}a_{i,k}$ is not identically 0. We choose such a $k$ and denote it by $k_\ast$. Then from \cite[Corollary 20]{Miy5}, $\sum_{i\in S_0}\Lambda_{H^\prime} a_{i,k_\ast}(x_m)$ is bounded above by a negative constant $-C$. Then from (\ref{j}), we have
\begin{align*}
\frac{1}{Cm}\geq \sigma_{S_0\setminus\{k_\ast\}}(x_m)-b_{S_0}\geq b_{S_0\setminus\{k_\ast\}}-b_{S_0}\geq 0.
\end{align*}
By taking the limit as $m\to\infty$, we obtain $b_{S_0\setminus\{k_\ast\}}=b_{S_0}$. By repeating the above argument, we can see that there exists a sequence of subsets $S_0\supsetneq S_1\supsetneq \cdots \supsetneq S_l=\emptyset$ such that $b_{S_0}=b_{S_1}=\cdots= b_{S_l}=0$. However, this contradicts the assumption $b_{S_0}<0$. Therefore, we see that $b_S\geq 0$ for any $S\subseteq \{0,1,\dots, r-1\}$. Suppose that $h$ and $h^\prime$ satisfy Condition \ref{three} in Section {2.2}. Since it is sufficient to lift $X$ to the universal covering space and show the theorem, we will assume from the beginning that $h$ and $h^\prime$ have approximations $(h_\epsilon)_{\epsilon>0}$ and $(h^\prime_\epsilon)_{\epsilon>0}$ in Condition \ref{three}. We denote by $H_{j,\epsilon}$ (resp. $H_{j,\epsilon}^\prime$) $(j=0, 1,\dots, r)$ the Hermitian metrics on $K_X^{-1}\rightarrow X$ constructed from $h$ and $e^{-\varphi}h_\refe$ (resp. $h^\prime$ and $e^{-\varphi^\prime}h_\refe$) as in Section \ref{2.2}. Moreover, we set $\sigma_{j,\epsilon}\coloneqq \log(H_{j,\epsilon}\otimes H_{j,\epsilon}^{\prime-1})$ for each $j=0,1,\dots, r$. Suppose that $e^{\varphi_\epsilon}h_\refe^{-1}$ is complete for each $\epsilon>0$. Then from what we have already proved, $\sigma_{j,\epsilon}\geq 0$ for each $j=0,1,\dots, r-1$. By taking the limit as $\epsilon\searrow 0$, we have $\sigma_j\geq 0$. Suppose that there exists $j\in\{1,\dots, r-1\}$ such that $\sigma_j$ is identically zero. Then from (\ref{sigmaj}), we see that both of $\sigma_{j-1}$ and $\sigma_{j+1}$ are identically zero. Suppose that $\sigma_0$ is identically zero. Then from (\ref{sigma0}), we see that $\sigma_1=\sigma_{r-1}$ is identically zero. Therefore, if we assume that $\varphi(x)>\varphi^\prime(x)$ holds at least one $x\in X$, then all of $\sigma_1,\dots, \sigma_r$ are not identically zero. We choose one from among $H_1,\dots, H_{[r/2]}$ and denote it by $H$. For each $j=1,\dots, r-1$, from (\ref{sigmaj}), we have
\begin{align*}
\inum\Lambda_H\partial\bar{\partial}\sigma_j\leq 2\Lambda_Hc_j\sigma_j\leq 2(\sup_X \Lambda_Hc_j)\sigma_j.
\end{align*} 
Note that $\sup_X \Lambda_H c_j$ is finite for the same reason that the constant $c$ defined in (\ref{constantc}) is finite. Let $x_0 \in X$ be an arbitrary point. Then from the mean value inequality theorem, there exists a positive constant $C$ and $0<p<1$ such that
\begin{align}
\sigma_j(x_0)\geq C\left(\int_{B_g(x_0,1/4)}\sigma_j^pd\mu_g\right)^{1/p}, \label{six0}
\end{align}
where $g$ is the K\"ahler metric induced by $H$. Let $A$ be the set $\{x\in X\mid \sigma_1(x)=0\}$. Inequality (\ref{ux0}) implies that the set $A$ is both closed and open. Since $\sigma_j$ is not identically zero, we have $\sigma_j>0$. Therefore, from the monotonicity of the logarithmic function, we have
\begin{align*}
F(r,\beta,\varphi^\prime, H_\refe)(x) - F(r,\beta,\varphi,H_\refe)(x) > 0 \quad \text{for each } x \in X.
\end{align*}
Moreover, by an elementary estimate for the log-sum-exp function, we have
\begin{align*}
F(r,\beta,\varphi^\prime, H_\refe)(x)- F(r,\beta,\varphi,H_\refe)(x) 
\leq &\max_{j\in\{1,\dots,r\}}\log(H_j\otimes H_\refe^{-1})(x)+\frac{1}{\beta}\log r+F(r,\beta,\varphi^\prime, H_\refe)(x) \\
=&\log(H_n\otimes H_\ast^{-1})(x)+\frac{1}{\beta}\log r+F(r,\beta,\varphi^\prime, H_\refe)(x) \ \text{for each $x\in X$}.
\end{align*}
where $n$ denotes $[r/2]$ and where in the final equality we have used \cite[Proposition 22]{Miy4}. Then we have
\begin{align*}
\log(H_n\otimes H_\ast^{-1})(x)+\frac{1}{\beta}\log r+F(r,\beta,\varphi^\prime, H_\refe)(x) 
<&\log(H_n\otimes H_\refe^{-1})(x)-\log(H_n^\prime\otimes H_\refe^{-1})(x)+\frac{1}{\beta}\log r \\
=&\log(H_n\otimes H_n^{\prime-1})(x)+\frac{1}{\beta}\log r \\
\leq &r(\varphi-\varphi^\prime)(x)+\frac{1}{\beta}\log r,
\end{align*}
where in the final inequality, we have used $H_0\otimes H_0^{\prime-1}\geq 1$. This completes the proof. 
\end{proof}

\subsection{Proof of Theorem \ref{main theorem a}} 
\begin{proof}[Proof of Theorem \ref{main theorem a}] Let $r$ be either $2$ or $3$. Let $h$ and $h^\prime$ be complete solutions to equation (\ref{phi}) associate with $\varphi$ and $\varphi^\prime$, respectively. We denote by $H_0, H_1$ (resp. $H_0^\prime, H_1^\prime$) the Hermitian metrics on $K_X^{-1}\rightarrow X$ constructed from $h$ and $e^{-\varphi}h_\refe$ (resp. $h^\prime$ and $e^{-\varphi^\prime}h_\refe$) as in Section \ref{2.3}. For each $j=0,1$, we set $\sigma_j\coloneqq \log(H_j\otimes H_j^{\prime-1})$ and $c_j\coloneqq \int_0^1\vol(H_j)^t\vol(H_j^\prime)^{1-t}dt$. We also set $\tau\coloneqq \log(H_0\otimes H_1^{-1}\otimes H_0^{\prime-1}\otimes H_1^\prime)=\sigma_0-\sigma_1$. We suppose that $h$ and $h^\prime$ satisfy Condition \ref{three} in Section \ref{2.3}. Since it is sufficient to lift $X$ to the universal covering space and show the statement of the theorem, we will assume from the beginning that $h$ and $h^\prime$ have an approximation $(h_\epsilon)_{\epsilon>0}$ and $(h_\epsilon^\prime)_{\epsilon>0}$ in Condition \ref{three}. We denote by $H_{0,\epsilon}, H_{1,\epsilon}$ (resp. $H_{0,\epsilon}^\prime, H_{1,\epsilon}^\prime$) the Hermitian metrics on $K_X^{-1}\rightarrow X$ constructed from $h_\epsilon$ and $e^{-\varphi_\epsilon}h_\refe$ (resp. $h^\prime_\epsilon$ and $e^{-\varphi^\prime_\epsilon}h_\refe$) as in Section \ref{2.3}. For each $j=0,1$, we set $\sigma_{j,\epsilon}\coloneqq \log(H_{j,\epsilon}\otimes H_{j,\epsilon}^{\prime-1})$ and $c_{j,\epsilon}\coloneqq \int_0^1\vol(H_{j,\epsilon})^t\vol(H_{j,\epsilon}^\prime)^{1-t}dt$. We also set $\tau_\epsilon\coloneqq \log(H_{0,\epsilon}\otimes H_{1,\epsilon}^{-1}\otimes H_{0,\epsilon}^{\prime-1}\otimes H_{1,\epsilon}^\prime)=\sigma_{0,\epsilon}-\sigma_{1,\epsilon}$. Suppose that $\varphi\geq \varphi^\prime, \inum\bar{\partial}\partial \varphi\geq \inum\bar{\partial}\partial \varphi^\prime$, and $e^{\varphi_\epsilon} h_\refe^{-1}$ is complete for each $\epsilon>0$. Then from the proof of Theorem \ref{main theorem b}, we have $\sigma_{0,\epsilon}\geq 0$, $\sigma_{1,\epsilon}\geq 0$ for each $\epsilon>0$. Moreover, from (\ref{sigmaj}) and (\ref{sigma0}) in the proof of Theorem \ref{main theorem b}, we have 
\begin{align}
\inum\Lambda_{H_{1,\epsilon}^\prime}\partial \bar{\partial}\tau_\epsilon&=\inum \Lambda_{H_{1,\epsilon}^\prime}\partial \bar{\partial}\sigma_{0,\epsilon}-\inum \Lambda_{H_{1,\epsilon}^\prime}\partial \bar{\partial}\sigma_{1,\epsilon} \notag \\
&\leq 2\Lambda_{H_{1,\epsilon}^\prime}c_{0,\epsilon}\sigma_{0,\epsilon}-2\Lambda_{H_{1,\epsilon}^\prime}c_{1,\epsilon}\sigma_{1,\epsilon}-(4-r)\Lambda_{H_{1,\epsilon}^\prime}c_{1,\epsilon}\sigma_{1,\epsilon}+(4-r)\Lambda_{H_{1,\epsilon}^\prime}c_{0,\epsilon}\sigma_{0,\epsilon} \notag \\
&=(6-r)\Lambda_{H_{1,\epsilon}^\prime}c_{0,\epsilon}\sigma_{0,\epsilon}-(6-r)\Lambda_{H_{1,\epsilon}^\prime}c_{1,\epsilon}\sigma_{1,\epsilon}
\notag \\
&\leq (6-r)\Lambda_{H_{1,\epsilon}^\prime}c_{1,\epsilon}(\sigma_{0,\epsilon}-\sigma_{1,\epsilon}) \notag \\
&=(6-r)\Lambda_{H_{1,\epsilon}^\prime}c_{1,\epsilon}\tau_\epsilon, \label{tauepsilon}
\end{align}
where, in evaluating $c_{0,\epsilon} \leq c_{1,\epsilon}$, we have used \cite[Proposition 22]{Miy4}. From Lemma \ref{phicomp}, \cite[Corollary 20]{Miy5}, and the assumption that $e^{\varphi_\epsilon}h_\refe^{-1}$ is complete, $\tau_\epsilon$ is bounded below. Therefore, by the Omori-Yau maximum principle, we have $(x_m)_{m\geq m_0}$ on $X$ such that
\begin{align}
-\tau(x_m)&\geq \sup_X(-\tau)-1/m, \label{tauxm}\\
1/m&\geq -(6-r)\Lambda_{H_{1,\epsilon}^\prime}c_{1,\epsilon}(x_m)\tau_\epsilon(x_m). \label{6-r}
\end{align}
From \cite[Corollary 20]{Miy5}, $c_{1,\epsilon}$ is bounded from below by a positive constant. Therefore, from (\ref{tauxm}) and (\ref{6-r}), we have $\tau_\epsilon\geq 0$. By taking the limit as $\epsilon\searrow 0$, we have $\tau\geq 0$. From this and \cite[Lemma 27]{Miy4}, we have $S(r,\beta,\varphi)\geq S(r,\beta,\varphi^\prime)$. Suppose that $\varphi^\prime - \varphi$ belongs to $W^{1,2}_{loc}$ and $\varphi(x) > \varphi^\prime(x)$ holds at least at one point of $X$. From the same calculation as (\ref{tauepsilon}), we have
\begin{align*}
\inum\Lambda_{H_1}\partial \bar{\partial}\tau&\leq (6-r)\Lambda_{H_1}c_1\tau \\
&\leq (6-r)(\sup_X\Lambda_{H_1}c_1)\tau.
\end{align*}
The constant $\sup_X\Lambda_{H_1}c_1$ is finite since from \cite[Corollary 20]{Miy5}, $\Lambda_{H_{1,\epsilon}}c_{1,\epsilon}$ is bounded from above by a constant independent of $\epsilon$. Let $x_0 \in X$ be an arbitrary point. Then from the mean value inequality theorem, there exists a positive constant $C$ and $0<p<1$ such that
\begin{align}
\tau(x_0)\geq C\left(\int_{B_g(x_0,1/4)}\tau d\mu_g\right)^{1/p}, \label{tau0}
\end{align}
where $g$ is the K\"ahler metric induced by $H_1$. Let $A_\tau$ be the set $\{x\in X\mid \tau(x)=0\}$. Inequality (\ref{tau0}) implies that the set $A_\tau$ is both closed and open. Since $\varphi(x)>\varphi^\prime(x)$ holds at least one point of $X$, $\tau$ is not identically zero. Therefore, we have $\tau>0$. This implies $S(r,\beta,\varphi)>S(r,\beta,\varphi^\prime)$.
\end{proof}

\subsection{Proof of Theorem \ref{main theorem c}}
\begin{proof}[Proof of Theorem \ref{main theorem c}] First, we mention the conditions that are immediately seen to be equivalent. By definition, (b) and (c) are equivalent. Condition (d) is equivalent to $H_j \otimes h_X$ being bounded for all $j=1,\dots, r$, which, by \cite[Corollary 20]{Miy5}, is in turn equivalent to (a). Similarly, condition (e) is equivalent to $H_j \otimes h_X$ being bounded for some $j=1,\dots, r$, which, by \cite[Corollary 20 and Corollary 22]{Miy5}, is also equivalent to (a). 

We next show that (a) implies (b). We set $n\coloneqq [r/2]$. For each $j=1,\dots, n$ we define $\rho_j$ as $\rho_j\coloneqq 1-H_{j-1}\otimes H_j^{-1}$. Then from \cite[Proposition 22]{Miy4}, $\rho_0,\dots, \rho_n$ are nonnegative functions. Suppose that (a) holds. Then from \cite[Corollary 20]{Miy5}, $H_j\otimes h_X$ is bounded above for each $j=0,\dots, n$. We set $H\coloneqq h_X^{-1}$. Then, by a direct calculation (cf. \cite[Proof of Theorem 4.8]{DL3}), we obtain
\begin{align}
\inum\Lambda_H\partial\bar{\partial}\rho_1&\leq \inum \Lambda_H\partial \bar{\partial}(-\log(H_0\otimes H_1^{-1})) H_0\otimes H_1^{-1} \notag \\
&\leq \Lambda_H(-3\vol(H_0)+4\vol(H_1)-\vol(H_2)) H_0\otimes H_1^{-1}\notag \\
&=\Lambda_H(3(1-H_0\otimes H_1^{-1})\vol(H_1)+\vol(H_1)-\vol(H_2)))H_0\otimes H_1^{-1} \notag \\
&\leq 3(H_0\otimes h_X) \rho_1 \notag \\
&\leq 3\sup_X(H_0\otimes h_X) \rho_1, \label{rho0} \\
\inum\Lambda_H\partial\bar{\partial}\rho_j&\leq \inum \Lambda_H\partial \bar{\partial}(-\log(H_{j-1}\otimes H_j^{-1})) H_{j-1}\otimes H_j^{-1} \notag \\
&=\Lambda_H(3\vol(H_j)-3\vol(H_{j-1})+\vol(H_{j-2})-\vol(H_{j+1})) H_{j-1}\otimes H_j^{-1} \notag \\
&3\Lambda_H\vol(H_j)(1-H_{j-1}\otimes H_j^{-1}) H_{j-1}\otimes H_j^{-1} \notag \\
&=3(H_{j-1}\otimes h_X)\rho_j \notag \\
&\leq 3\sup_X(H_{j-1}\otimes h_X)\rho_j\ \text{for $j=2,\dots, n$},\label{rhoj}
\end{align}
where we have used $H_1\otimes H_2^{-1}\leq 1$ and $H_{j-2}\otimes H_j^{-1}\leq 1$ for each $j=2,\dots, n$ (see \cite[Proposition 22]{Miy4}). By applying Lemma \ref{delta} to inequality (\ref{rhoj}), we obtain $\rho_j\geq \delta_j$ for each $j=2,\dots, n$ with some positive constants $\delta_2,\dots, \delta_n$. Moreover, if $e^{\varphi}h_\refe^{-1}$ belongs to $W^{1,2}_{loc}$, then by applying Lemma \ref{delta} to inequality (\ref{rho0}), we obtain $\rho_1\geq \delta_1$ for some positive constant $\delta_1$. As mentioned in Section \ref{2.5}, the Shannon entropy achieves its maximum value $\log r$ only when the probability distribution is uniform. Therefore, when $r\geq 4$, the above result concerning $\rho_2,\dots, \rho_n$ implies that there exists a positive constant $\delta$ such that $S(r,\beta,\varphi)\leq \log r-\delta$. When $r=2,3$, under the assumption that $e^\varphi h_\refe^{-1}$ belongs to $W^{1,2}_{loc}$, the estimate $\rho_1\geq \delta_1$ implies that there exists a positive constant $\delta$ such that $S(r,\varphi,\beta)\leq \log r-\delta$. Therefore, condition (b) holds.

We then show that (b) implies (f). Suppose that (b) holds. Moreover, when $r=2,3$, we assume that $e^\varphi h_\refe^{-1}$ belongs to $W^{1,2}_{loc}$. Since the Shannon entropy achieves its maximum value $\log r$ only when the probability distribution is uniform, there exists $0\leq j<k\leq n$ and $\delta>0$ such that $H_j\otimes H_k^{-1}=p_j(r,1,\varphi)/p_k(r,1,\varphi)<1-\delta$, where $p_j(r,\beta,\varphi)$ is defined in Definition \ref{entropy2}. From \cite[Proposition 22]{Miy4}, we have $H_{j-1}\otimes H_j<1$ for each $j=0,1,\dots, n$. Therefore, there exists $1\leq k\leq n$ such that $H_{k-1}\otimes H_k^{-1}<1-\delta$. From Proposition \ref{FEineq}, for each positive $\beta$, we have
\begin{align}
\inum\partial\bar{\partial} (F(\beta, r, \varphi, H_\refe)-F(\beta,r,-\infty, H_\refe))\leq -\frac{\sum_{j=1}^r(\vol(H_{j-1})-\vol(H_j))(\vol(H_{j-1})^\beta-\vol(H_j)^\beta)}{\sum_{j=1}^r\vol(H_j)^\beta}+\omega_X. \label{omegax}
\end{align}
From \cite[Corollary 22]{Miy5}, $H_1,\dots, H_n$ are mutually bounded, and $H_0\otimes H_j^{-1}$ is bounded for each $j=1,\dots, n$. Therefore, there exists a positive constant $C$ such that $\sum_{j=1}^r\vol(H_j)^\beta\leq C\vol(H_k)^\beta$. Then we obtain
\begin{align}
-\frac{\sum_{j=1}^r(\vol(H_{j-1})-\vol(H_j))(\vol(H_{j-1})^\beta-\vol(H_j)^\beta)}{\sum_{j=1}^r\vol(H_j)^\beta}\leq &-\frac{(\vol(H_{k-1})-\vol(H_k))(\vol(H_{k-1})^\beta-\vol(H_k)^\beta)}{C\vol(H_k)^\beta} \notag \\
= &-\frac{1}{C}(1-H_{k-1}\otimes H_k^{-1})(1-(H_{k-1}\otimes H_k^{-1})^\beta)\vol(H_k) \notag \\
\leq &-\frac{1}{C}\delta(1-(1-\delta)^\beta)\vol(H_k) \notag \\
\leq &-\frac{1}{C}\delta(1-(1-\delta)^\beta)C^{1/\beta}(\sum_{j=1}^r\vol(H_j)^\beta)^{1/\beta} \notag \\
= &-C_3e^{-F(\beta,r,\varphi,H_\refe)}\vol(H_\refe) \label{hrefe},
\end{align}
where $C_3$ is defined by $C_3\coloneqq \frac{1}{C}\delta(1-(1-\delta)^\beta)C^{1/\beta}$. From (\ref{omegax}) and (\ref{hrefe}), we obtain the claim.

Finally, we show that (f) implies (d). Suppose that (f) holds. Then there exists a positive constant $C_3$ such that 
\begin{align}
\inum \partial \bar{\partial} (F(r,\beta,\varphi, H_\refe)-F(r,\beta,-\infty, H_\refe)) \\
\leq -C_3 e^{-F(r,\beta,\varphi, H_\refe)}\vol(H_\refe)+\omega_X. \label{F-F}
\end{align}
From (\ref{F-F}), we obtain
\begin{align}
\inum \Lambda_H\partial \bar{\partial} (-F(r,\beta,\varphi, H_\refe)+F(r,\beta,-\infty, H_\refe))\geq C_3\Lambda_He^{-F(r,\beta,-\infty, H_\refe)}\vol(H_\refe) e^{-F(r,\beta,\varphi, H_\refe)+F(r,\beta,-\infty, H_\refe)}+1, \label{final inequality}
\end{align}
where $H$ is the Hermitian metric on $K_X^{-1}\rightarrow X$ induced by the Poincar\'e metric. Note that $\Lambda_He^{-F(r,\beta,-\infty, H_\refe)}\vol(H_\refe)$ is a positive constant. Let $Y\subseteq X$ be an open subset with smooth boundary such that $K\subseteq Y$ and the closure $\overline{Y}$ is compact. By applying the Cheng-Yau maximum principle for functions on manifolds with boundary to (\ref{final inequality}), we see that $-F(r,\beta,\varphi, H_\refe)+F(r,\beta,-\infty, H_\refe)$ is bounded above. This implies (b), and therefore we obtain the claim.

\end{proof}

\noindent
E-mail address 1: natsuo.miyatake.e8@tohoku.ac.jp

\noindent
E-mail address 2: natsuo.m.math@gmail.com \\

\noindent
Mathematical Science Center for Co-creative Society, Tohoku University, 468-1 Aramaki Azaaoba, Aoba-ku, Sendai 980-0845, Japan.

\end{document}